\documentclass[graybox]{svmult}

\usepackage{mathptmx}       
\usepackage{helvet}         
\usepackage{courier}        
\usepackage{type1cm}        
\usepackage{makeidx}         
\usepackage{graphicx}        
\usepackage{multicol}        
\usepackage[bottom]{footmisc}

\makeindex             

\usepackage{amsmath,amssymb}
\usepackage{amsfonts}

\newcommand{\pa}{{\partial}}

\DeclareMathOperator*{\Ad}{Ad}
\DeclareMathOperator*{\ad}{ad}

\def\0{{\bf 0}}

\def\d{{\mathrm{d}}}

\def\R{\mathbb{R}}

\def\v{{\bf v}}
\def\w{{\bf w}}

\def\lieg{\mathfrak{g}}

\newif\iftodo
\todotrue

\iftodo
\newcommand{\todo}[1]{\vspace{5 mm}\par \noindent
\marginpar{\textrm{ToDo}} \framebox{\begin{minipage}[c]{0.45
\textwidth} \tt #1 \end{minipage}}\vspace{5 mm}\par}
\else
\newcommand{\todo}[1]{}
\fi

\begin{document}

\title*{Normal forms for Lie symmetric cotangent  bundle systems with free and proper actions}
\titlerunning{Normal forms for rotationally-invariant systems} 
\author{Tanya Schmah and Cristina Stoica}
\institute{Tanya Schmah \at University of Toronto, \email{schmah@cs.toronto.edu}
\and Cristina Stoica \at Wilfrid Laurier University, \email{cstoica@wlu.ca}}
%
%
\maketitle

\abstract{
We consider free and proper cotangent-lifted symmetries of Hamiltonian systems.
For the special case of $G = SO(3)$, we construct symplectic slice coordinates
around an arbitrary point.
We thus obtain a parametrisation of the phase space suitable for  the study of dynamics near relative equilibria, in particular for the Birkhoff-Poincar\'e normal form method.
For a general symmetry group $G$, we observe that for the calculation of the  truncated normal forms, one does not need an
explicit coordinate transformation but only its higher derivatives at the relative equilibrium.
 We outline an iterative scheme using these derivatives for the computation of truncated Birkhoff-Poincar\'e normal  forms. 
%
%
%
%
%
%
}

\tableofcontents

\section{Introduction}


The Birkhoff-Poincar\'e normal  form  is one of the main tools used in studying local bifurcation and stability for dynamical systems. It is a method  based on applying coordinate transformations that simplify the jets of a vector field at an equilibrium, up to a certain order.  
For Hamiltonian vector fields, the transformations applied must be symplectic, or more generally Poisson, so that the truncated vector field preserves its structure.  
We will not  report here on the importance and usefulness  of  normal forms in relation, for instance,  to bifurcation and stability theory; the interested reader may consult, for instance, \cite{Broer09} and references therein, as well as \cite{M92}. 


 For Lie symmetric  systems, relative equilibria play an important r\^ole  in dynamics, analogous 
 to the r\^ole that equilibria play for generic vector fields.
  A very common first step in dynamical studies near relative equilibria 
is to use a \textit{slice} theorem to pass to a coordinate system 
that separates directions along and transversal to the group orbit.
%
%
%
Indeed,  if the symmetry group $G$ acts freely and properly, then in a sufficiently small neighborhood   of an orbit $Gz_0$,
the phase space is isomorphic to 
the \textit{slice bundle}  $G \times S$ where $S$, called the \textit{slice},  is a subspace of the tangent space at $z_0$ that is transversal to $Gz_0$.
(For the non-free case, as well as a characterisation of normal forms 
near relative equilibria, see \cite{L07}.) %
This local model of the action of $G$ on the phase space is actually ``semi-global''  in the sense that it is
global ``in the $G$ direction'' but local ``in the transverse direction''. 

In these coordinates,  a relative equilibrium $z(t)$ corresponds to $\left(\text{exp}(t \omega), 0 \right) $ and the dynamics takes the form 
 \[\dot g= g f_G(s), \quad \quad \dot s=  f_S(s)\,,\]
where $f_S: S\to S$ and $f_G: S \to \mathfrak{g}.$ 
  Thus, locally, the dynamics in the slice $\dot s=  f_S(s)$ drives the dynamics in the group (or ``drift") directions $\dot g= g f_G(s)$. 

\smallskip
In the case of Lie symmetric Hamiltonian systems, the dynamics may be split into the ``drift'' and ''slice'' directions as above, but it must also accommodate the additional Hamiltonian structure. By Noether's theorem, the symmetry group $G$ provides the Hamiltonian system with conserved quantities, called momenta. The symplectic manifold is therefore partitioned into flow-invariant level sets of the momenta. The way these level sets intersect the slice can be complicated, especially when $G$ is non-abelian and the relative equilibrium has non-trivial isotropy (the \textit{non-free action} case). This leads to a nontrivial structure on the slice bundle, which in turn induces a nontrivial structure on the slice equations, \cite{RWL02, RSS06}. 

For calculating Birkhoff-Poincar\'e normal forms near relative equilibria of Hamiltonian systems, a natural approach   is to   try to transfer  the machinery from the case of canonical Hamiltonian systems near an equilibrium. 
For many dynamical studies, it is sufficient to consider a single symplectic reduced space at a
single momentum level $\mu_0$,
 in which the original relative equilibrium $z_0$ corresponds to an equilibrium.
 
In the case of cotangent-bundle systems $T^*Q$, coordinates on the symplectic reduced space suitable for the normal form computation may be found by
applying a slice theorem in the configuration space: $Q$ is locally modelled as $G \times S$,
where $S$ is an ``internal-shape space" direction transverse to the group orbit $Gq_0$.
Then the symplectic reduced space may be identified with $\mathcal{O}_{\mu_0} \times T^*S$ (where 
$\mathcal{O}_{\mu_0}$ is the coadjoint orbit through $\mu_0$), with the 
KKS and canonical symplectic forms.
This point of view is pursued in \cite{Ciftci12} and \cite{Ciftci14}; see also \cite{MoRo99}.

However it is not always sufficient to consider only a single symplectic reduced space, or a single momentum level set. In particular, the analysis of symmetry-breaking perturbations requires 
symmetry-adapted local coordinates for the entire phase space that simultaneously
place the group action, its momentum map and the symplectic form in simple forms.
For symplectic actions, the   Hamiltonian slice theorem of Marle \cite{Mar85}
and Guillemin and Sternberg \cite{GS84} (see Theorem \ref{marle} below)
achieves this goal, modelling the phase space as $G\times \mathfrak{g}_{\mu_0}^{\ast}\times N_{s}$
(for free actions), where $N_s$ is the symplectic normal space.
In these symplectic slice coordinates, 
the momentum level set $J^{-1}(\mu_0)$
becomes $G_{\mu_0} \times \{0\} \times N_s$, 
where $G_{\mu_0}$ is the isotropy group of $\mu_0$ with respect to the co-adjoint action;
so the corresponding symplectic reduced space may be symplectically embedded in the unreduced space
as $\{e,0\} \times N_s$.
For free actions 
on cotangent bundles,
$N_s \cong T_{\mu_0} \mathcal{O}_{\mu_0} \times T^*S$,
for $S$ a slice in configuration space, with KKS and canonical symplectic forms,
which illustrates the connection with the approach in the previous paragraph.

The applicability of the Hamiltonian slice theorem is obstructed by the lack of a constructive proof. Practically,  one does not know the change of coordinates, known as the \textit{symplectic tube}, which renders the desired structure. This explicit change of coordinates has been found in only two special cases:
(i) cotangent-lifted actions where $G_{\mu_0} = G$ (this happens, for instance, at zero momentum and for abelian Lie groups) \cite{Sch07}; and (ii) free cotangent-lifted actions of $G=SO(3)$, which we present here in Sections \ref{ssect:sliceSO3} and \ref{ssect:arbitrary}.

\medskip In this paper we outline an algorithm for the computation of truncated Birkhoff-Poincar\'e normal  forms  near a relative  equilibrium, for Lie symmetric  cotangent bundle systems with free and proper actions. The splitting of the phase space and the associated change of coordinates are explicitly given for 
$SO(3)$-symmetric systems.
The general algorithm, for any symmetry group $G$,
is based on an iterative scheme which allows the calculation of the truncated normal form up to any desired order. At its core, our method relies on the observation that for   the calculation of the  truncated normal forms one does not need an
explicit coordinate transformation
but only its derivatives at the equilibrium/relative equilibrium. We thank Mark Roberts 
making this key observation in a discussion about 10 years ago.

  For the dynamics on $T^*SO(3)$, the coordinates we have obtained for the reduced space coincide with the \textit{regularised}  Serret-Andoyer-Deprit coordinates used in celestial mechanics (see \cite{Fasso_4} and references therein);  however, we retrieved these  coordinates  via a different path and this was crucial for arriving at a methodology for the general case. 
Our slice parameterisation uses a global description for the reconstruction (attitude) variable $R(t) \in G.$ (We use the word ``attitude'' in analogy to its use in rigid body dynamics.) In concrete applications, it is likely that a local coordinate system will be used. For example, for $SO(3),$ if an explicit local coordinate system
is sought, Serret-Andoyer-Deprit is probably the best choice,
because they are action-angle coordinates with a very simple relation to Euler angles.  

 We also compare the splitting of the phase space used for the computations of the normal forms with those used in the Reduced Energy Momentum Method (REM) \cite{SLM91, M92}, 
 the latter citation being to Jerry Marsden's ``blue book''. We respond to one of Jerry's questions stated on page 104 of that book:

\medskip
 \textit{It is also of interest to link the normal forms here} (i.e., in the REM) \textit{with those in singularity theory. In particular, can one use the forms here as first terms in higher order normal forms?}
 
 \medskip
 In short,  the answer is no: while the REM splittings are very useful when looking for sufficient conditions for  stability with minimal computational effort, 
 they do not organise the symplectic form in a convenient form for the  Birkhoff-Poincar\'e normal form method. We expand on this subject in Section \ref{sect:REM}.

\medskip
This paper is organised as follows. In Section 2 we investigate the free action of a Lie group $G$ on $T^*G$ by cotangent lifts,
arriving at a general \textit{Tube Condition} given in Proposition \ref{tubecond}.
(Some technical details from this section appear in the Appendix.) 
We then focus on the special case of $G = SO(3)$ and succeed in constructing
an explicit symplectic tube around an arbitrary point, see Theorem \ref{SO3tube}.
We use this to construct a symplectic tube
for any free cotangent-lifted action of $SO(3)$ on an arbitrary manifold, see Section \ref{ssect:arbitrary}.
In Section 3 we offer the equations of motion in slice  coordinates  for dynamics on $T^*SO(3)$ and for cotangent bundle rotationally invariant systems, including the case of simple mechanical systems. 
In Section 4 we outline the  algorithm for calculating truncated Birkhoff-Poincar\'e normal forms for general free and proper actions. Section 5 comments on the relationship between the splittings used in these normal forms and those in the Reduced Energy Momentum method.

\section{Slice coordinates}

\subsection{Lie symmetries of Hamiltonian systems}

For general background information on Lie symmetries,
see \cite{HoSchSt09}.
In what follows,  gothic letters will always denote Lie algebras of the Lie groups with
corresponding latin letters.
Let $G$ act on $M$, with the action of $g\in G$ on $z\in M$ denoted by $gz$. 
The corresponding infinitesimal action of $\xi \in \mathfrak{g}$ on $z$ is denoted by $\xi z$.
The \emph{isotropy subgroup} of a point $z\in M$ is $G_{z}:=\left\{ g\in G\mid g z=z\right\} .$
The adjoint action of $G$ on $\mathfrak{g}$ is denoted by $\Ad$,
and the infinitesimal adjoint action by $\ad$.
The \emph{coadjoint} action of $G$ on $\mathfrak{g}^{\ast}$
is the inverse dual to the adjoint action, 
$g\nu=\Ad_{g^{-1}}^{\ast}\nu$.
The infinitesimal coadjoint action  is given by
$\xi\cdot \nu=-\ad_{\xi}^{\ast}\nu$. 
For any $\mu\in \mathfrak{g}^*$, the notation $G_\mu$ will always
denote the isotropy subgroup of $G$ with respect to the coadjoint action, that is
$G_\mu:=\{g \in G\,|\, \Ad_{g^{-1}}^* \mu=g\}.$
The notation introduced is summarised in the following table.

\bigskip

\centerline{
\begin{tabular}{|r|l|}
\hline
$\mathfrak{g}$ & Lie algebra of a Lie group $G$ \\ \hline
$gz$ & action of $g\in G$ on $z$ \\ \hline
$\xi z$  & infinitesimal action of $\xi \in \mathfrak{g}$ on $z$ \\ \hline
$\Ad_{g}$ & adjoint action of $g\in G$ on $\mathfrak{g}$ \\ \hline
$ \Ad^*_{g^{-1}}$ & coadjoint action of $g\in G$ on $\mathfrak{g}^*$ \\ \hline
$\ad_{\xi}$ & infinitesimal adjoint action of $\xi \in \mathfrak{g}$ on $\mathfrak{g}$ \\ \hline
$ -\ad^*_{\xi}$ & infinitesimal coadjoint action of $\xi \in \mathfrak{g}$ on $\mathfrak{g}^*$ \\ \hline
$G_\mu$ & isotropy subgroup of $\mu$ w.r.t. coadjoint action\\
\hline
\end{tabular}
}

\bigskip
Suppose $G$ acts symplectically on a symplectic manifold $\left(M,\Omega\right).$ 
Recall that
any function $F:M\rightarrow\mathbf{R}$ defines a Hamiltonian vector
field $X_{F}$ by $i_{X_{F}}\Omega=dF,$ in other words $\Omega\left(X_{F}\left(z\right),v\right)=dF\left(v\right)$
for every $v\in T_{z}^{\ast}M.$ A \emph{momentum map} is a function
$J:M\rightarrow\mathfrak{g}^{\ast}$ satisfying $X_{J_{\xi}}(z) = \xi z$
for every $\xi\in\mathfrak{g}$ and $z\in M$, where $J_{\xi}:M\rightarrow\mathbf{R}$
is defined by $J_{\xi}\left(z\right)=\left\langle J\left(z\right),\xi\right\rangle .$
If the
$G$ action has an $\Ad^{\ast}$-equivariant momentum
map $J,$ then it is called \emph{globally Hamiltonian.}

\medskip
The \textit{coadjoint orbit} through any $\mu \in \mathfrak{g}^*$ is the orbit of $\mu$
with respect to the coadjoint action, $\mathcal{O}_\mu := \{ \Ad_{g^{-1}}^* \mu : g \in G\}$.
The \textit{Kostant-Kirillov-Souriau (KKS)} symplectic forms on any coadjoint orbit
$\mathcal{O}_\mu$
are
given by
\begin{align}
\label{E:KKS}
\Omega_{\mathcal{O}_\mu}^{\pm}\left(  \nu\right)  \left(-\ad^*\nolimits_{\eta_1} \nu, -\ad^*\nolimits_{\eta_2} \nu\right) =\pm\left\langle \nu,\left[  \eta_1,\eta_2\right]
\right\rangle \, .
\end{align}
The momentum map of the coadjoint action of $G$ on $\mathcal{O}_\mu$
with respect to $\Omega_{\mathcal{O}_\mu}^{\pm}$
(the ``$\pm$KKS forms'')
is
$J_{\mathcal{O}_\mu}\left(\nu\right) = \pm \nu.$
It can be shown that the KKS forms
are always $G$-invariant.

Let $N_s$ be the \emph{symplectic normal space} at $z$,
\[
N_s(z):= \ker dJ(z)/\mathfrak{g}_{\mu} z\, ,
\]
where $\mathfrak{g}_{\mu} z := \{\xi z : \xi \in \mathfrak{g}_{\mu}\}$
The restriction of $\Omega(z)$ to $\ker dJ(z)$
has kernel $\mathfrak{g}_\mu z$,
by the Reduction Lemma \cite{AM78}, 
so it descends to 
a reduced symplectic bilinear form on $N_s(z)$.
For free and proper actions, this space is isomorphic to the tangent at $[z]$ to the symplectic reduced space $J^{-1}(\mu) / G_\mu$
(see \cite{M92}).

We now give limited versions of Palais' slice theorem \cite{Pal61, OR04} and the Hamiltonian
Slice Theorem of Marle, Guillemin and Sternberg \cite{Mar85, GS84}, 
treating only the case of free actions (for ease of exposition).

\begin{theorem}\label{Palais}
[``Palais' Slice Theorem'' for free actions]\cite{Pal61, OR04} Let \(G\) be a Lie group acting properly, smoothly and freely on a manifold \(M,\) and let \(z\in M.\) Choose a local Riemannian metric around \(z\)
(such a metric always exists),  let \(N\) be the orthogonal complement to
\(\mathfrak{g} z\), and let \(\exp_{z}\) be the corresponding
Riemannian exponential based at \(z.\) Then there exists a
neighbourhood \(S\) of \(0\) in \(N\) such that the map
\begin{align*}
\tau:G\times (S \subset N)  &  \rightarrow M\\
\left(  g,s\right)   &  \longmapsto g \exp_{z}s
\end{align*}
is a $G$-equivariant diffeomorphism.
(Such a $\tau$ is called a \textit{tube}.)
If $M$ is a vector space and $G$ acts linearly,
then the ``$exp_z s$'' in the formula for $\tau$ may be replaced by ``$z+s$'',
and $\tau$ is a $G$-equivariant diffeomorphism for any choice of an $H$-invariant neighbourhood $S$ of
$0$ for which $\tau$ is injective.
\end{theorem}

\medskip
Suppose that $G$ acts 
\emph{symplectically} on a manifold $\left(M,\Omega\right)$,
with $\Ad^{\ast}$\emph{-}equivariant
momentum map $J$.
We would like to find a \textit{symplectic} tube $\tau: G \times N \to M$, for some $N$,
with respect to some simple or ``natural'' symplectic form on $G\times N$.
The Hamiltonian Slice Theorem, 
also known as the
Marle-Guillemin-Sternberg normal form \cite{GS84, BL97}, accomplishes this, 
for actions that are not necessarily free.
We present the theorem now only for free actions.
Let $z\in M$ and 
$\mu=J\left(z\right)$, and let $G_{\mu}$ be the isotropy group of $\mu$ with respect to the coadjoint action.
(Note that $\mu$ is a specific momentum value, corresponding to the $\mu_0$ in the Introduction;
we have dropped the subscript $0$ for ease of notation.)
Let $N_s$ be the \emph{symplectic normal space} at $z$.
We define a symplectic form on 
$G\times \left(\mathfrak{g}_\mu^{\ast}\times N_{s}\right)$.
First, choose a specific $G$-invariant splitting $\lieg = \lieg_\mu \oplus \lieg_\mu^\perp$. 
Define $\Omega_{T}$ and $\Omega_\mu$ on $G\times\mathfrak{g}_{\mu}^{\ast}$ by 
\begin{align*}
\Omega_T(g,\nu)\left((\xi_1, \dot \nu_1), (\xi_2, \dot \nu_2)\right) 
&=  <\mu, [\xi_1,\xi_2]>  \\
\Omega_0 (g,\nu)\left((\xi_1, \dot \nu_1), (\xi_2, \dot \nu_2)\right) 
&=  \left<\nu, \left[\xi_1,\xi_2\right]\right\rangle +
 <\dot \nu_2, \xi_1^\mu> - <\dot \nu_1, \xi_2^\mu>,
\end{align*}
where $\xi_1^\mu$ and $\xi_2^\mu$ are the $\lieg_\mu$ components of $\xi_1$ and $\xi_2$.
Third, let $\Omega_{N_{s}}$ be the reduced symplectic bilinear form
on $N_{s}$ (defined above).
Then $\Omega_{Z}$ $:=\Omega_{T}+\Omega_{0}+\Omega_{N_{s}}$
is a presymplectic form on $G\times\mathfrak{g}_{\mu}^{\ast}\times N_{s}$. 
It can be shown that there exists
a $G$-invariant neighbourhood $Y$ of $\left[e,0,0\right]$ in
$G\times \mathfrak{g}_\mu^{\ast}\times N_{s}$
in which $\Omega_{Z}$ is symplectic. 
Let $\Omega_Y$ be the restriction of $\Omega_Z$ to $Y$.
Finally, note that there is
left $G$-action on $Y$ given by $
g^{\prime}\left(g,\nu,\rho\right)=\left(g^{\prime}g,\nu,\rho\right)$
It is easy to check that this is symplectic %
with respect to $\Omega_{Y}.$
 
\begin{theorem}\label{marle}
[Hamiltonian Slice Theorem for free actions]\footnote{
In the full Hamiltonian Slice Theorem, at a point $z$ with non-trivial isotropy group $G_z$, the model space is $G\times_{G_z} \left(  \mathfrak{g}_\mu^{\ast}\times N_{s}\right) $, and $J_Y$ has
an extra term.}\cite{GS84, BL97}
In
the above context, there exists a symplectic tube from
\(Y\subset G\times \mathfrak{g}_\mu^{\ast}\times N_{s}  \) to $M$
that maps 
\(\left(  e,0,0\right)\) to \(z.\)
The momentum map 
of the $G$ action on $Y$
is 
\[
J_Y\left(g,\nu,\rho\right) = {\Ad}^*_{g^{-1}}\left( \mu + \nu \right).
\]
\end{theorem}

\medskip

No general constructive proof of this theorem is known, even for free actions. However a constructive proof is given in \cite{Sch07} for the special case of a cotangent-lifted action, not necessarily free, for which $G_\mu = G$.

\subsection{Symplectic slices for the cotangent bundle of a Lie group}\label{ssect:cotbunslice}

We now consider the special case of $G$ acting on $T^*G$ by the cotangent lift of 
left multiplication. 
We left-trivialise $T^*G$, meaning that we identify it with $G \times \lieg^*$
via the map $p \in T_gG \mapsto (g,\mu) := (g, g^{-1} p)$, with $g^{-1}p := D\Phi_g(e)^* (p)$.
We seek a constructive symplectic tube based at a general $(g,\mu) \in G \times \lieg^*$ satisfying the conditions of the Hamiltonian Slice Theorem  (Theorem \ref{marle}).
Without loss of generality, we will assume $g = e$ (the identity). 
We will also assume that $\mu \ne 0$, since in the case $\mu = 0$ we have $G_\mu = G$
and a trivial symplectic normal space, so Theorem \ref{marle} is trivial.

Using left-trivialisation,
the canonical symplectic form becomes:
\begin{align*}
\Omega_c(e,\mu)\left((\xi_1, \rho_1), (\xi_2, \rho_2)\right) 
&= <\mu, [\xi_1,\xi_2]> + <\rho_2, \xi_1> - <\rho_1, \xi_2> .
\end{align*}
The $G$ action on $T^*G$ becomes $h (g,\nu) = (hg, \nu)$, 
which has momentum map
$J(g,\nu) = \Ad\nolimits_{g^{-1}}^* \nu$.
Note that $J$ is $\Ad^*$-equivariant) and
\begin{align}\label{E:DJ}
DJ(e,\mu)\cdot (\xi,\rho) = \xi \cdot \mu + \rho = -\ad\nolimits_\xi^*\mu + \rho.
\end{align}

\medskip

Fix a $\mu \in \lieg^*$, $\mu \ne 0$. Choose a $G$-invariant Riemannian metric on $\lieg$, and let $\lieg_\mu^\perp$
be the orthogonal complement of $\lieg_\mu$. 
Define
\begin{align}\label{E:N1}
N_1 := \left\{
\left(\eta, \ad\nolimits_\eta^*\mu\right) : \eta \in \lieg_\mu^\perp\right\}.
\end{align}
It follows from \eqref{E:DJ} that
$N_1$ is a complement to $\lieg_\mu z$ in $\ker DJ(z)$.
Therefore $N_1$ is isomorphic to the
symplectic normal space $N_s(z)$, with the reduced symplectic bilinear form on $N_s(z)$
corresponding to the restriction of $\Omega_c(z)$ to $N_1$.

\begin{lemma} \label{N1KKS}
The following is a linear symplectomorphism 
from $N_1$ (with the restricted canonical symplectic form)
to $T_\mu \mathcal{O}_\mu$ with the
KKS form $\Omega^-_\mathcal{O}(\mu)$,
\[
L:\left(\eta, \ad\nolimits_\eta^*\mu\right) \mapsto \ad\nolimits_\eta^*\mu\, .
\]
\end{lemma}

\begin{proof}
\begin{align*}
\Omega_c
\left(\left(\eta_1, \ad\nolimits_{\eta_1}^*\mu\right),
\left(\eta_2, \ad\nolimits_{\eta_2}^*\mu\right)\right)
&= <\mu, [\eta_1,\eta_2]> + <\ad\nolimits_{\eta_2}^*\mu, \eta_1> - <\ad\nolimits_{\eta_1}^*\mu, \eta_2> \\
&=-\left<\mu, \left[\eta_1, \eta_2\right]\right> 
= \Omega^-_{\mathcal{O}} (\mu)
\left(\ad\nolimits_{\eta_1}^*\mu,
\ad\nolimits_{\eta_2}^*\mu\right)\\
&=
\left(L^* \Omega^-_{\mathcal{O}} (\mu) \right)
\left(\left(\eta_1, \ad\nolimits_{\eta_1}^*\mu\right),
\left(\eta_2, \ad\nolimits_{\eta_2}^*\mu\right)\right).
\end{align*}
The result follows by equivariance of $L$ and invariance of the two symplectic forms.
\end{proof}

We identify $N_s(z) \cong N_1 \cong T_\mu \mathcal{O}_\mu$ via this lemma,
so that the reduced symplectic form $\Omega_{N_s}$ is identified with
$\Omega^-_\mathcal{O}(\mu)$.

\medskip

We seek a constructive version of the Hamiltonian Slice Theorem (for free actions) in this context.
That is, we wish to construct a $G$-equivariant local diffeomorphism 
\begin{align*}
\Phi: G \times \mathfrak{g}_\mu^* \times N &\longrightarrow G\times \lieg^*, \\
(e,0,0) &\mapsto (e,\mu),
\end{align*}
such that 
$\Phi^*\Omega_c = \Omega_Y:=\Omega_T + \Omega_0  + \Omega_N$, where
\begin{align} \label{E:OmegaY}
\Omega_T(g,\nu,\rho)\left((\xi_1, \dot \nu_1, \dot \rho_1), (\xi_2, \dot \nu_2, \dot \rho_2)\right) 
&=  <\mu, [\xi_1,\xi_2]> ,\\
\Omega_0(g,\nu,\rho)\left((\xi_1, \dot \nu_1, \dot \rho_1), (\xi_2, \dot \nu_2, \dot \rho_2)\right) 
&= \left<\nu, \left[\xi_1,\xi_2\right]\right\rangle +
 <\dot \nu_2, \xi_1^\mu> - <\dot \nu_1, \xi_2^\mu>,\notag \\ 
\Omega_N(g,\nu,\rho)\left((\xi_1,\dot \nu_1, \ad\nolimits_{\zeta_1}^* \mu ), 
(\xi_2,\dot \nu_2, \ad\nolimits_{\zeta_2}^* \mu )\right)
&= -\left<\mu, \left[\zeta_1, \zeta_2\right]\right>
. \notag
\end{align}

\medskip

The following proposition, proven in the Appendix, characterises
the symplectic tubes that appear in the Hamiltonian Slice Theorem (Theorem \ref{marle}).

\begin{proposition}[Tube Condition]\label{tubecond}
$\Phi^* \Omega_c = \Omega_Y$ if and only if
\begin{align*}
\Phi(g, \nu, \ad\nolimits_\eta^* \mu) 
= \left(
g F(\nu, \eta)^{-1}, 
\Ad\nolimits_{F(\nu, \eta)^{-1}}^*
\left(\mu + \nu\right)
\right)
\end{align*}
for some $F:\mathfrak{g}_\mu^* \times \mathfrak{g}_\mu^\perp \to G$ such that $F(0,0) = e$ and
\begin{align*}
&\left\langle 
\mu + \nu, 
\left[ 
F(\nu, \eta)^{-1} \left(DF(\nu, \eta) \cdot \left(\dot \nu_1, \zeta_1\right)\right),
F(\nu, \eta)^{-1} \left(DF(\nu, \eta) \cdot \left(\dot \nu_2, \zeta_2\right)\right)
\right]
\right\rangle \\
&+
\left\langle
\dot \nu_2, F(\nu, \eta)^{-1} \left(DF(\nu, \eta) \cdot \left(\dot \nu_1, \zeta_1\right)\right)
\right\rangle
- \left\langle
\dot \nu_1, F(\nu, \eta)^{-1} \left(DF(\nu, \eta) \cdot \left(\dot \nu_2, \zeta_2\right)\right)
\right\rangle \\
&= \left\langle \mu, \left[\zeta_1, \zeta_2\right]\right\rangle. 
\end{align*}
\end{proposition}

We have not found a general construction for a symplectic tube valid for all Lie groups $G$,
and indeed we do not expect that one will ever be found.
However we noticed, as explained in the Appendix, 
that the restriction of the Tube Condition to the subspace $\{0\} \times \{0\} \times N$
is reminiscent of the condition in the following lemma, which is proven in the Appendix.

\begin{lemma}\label{preserveKKS}
Let $\varphi:T_\mu\mathcal{O}_\mu \to \mathcal{O}_\mu$ 
be of the form $\varphi(-\ad\nolimits^*_\eta \mu) = f(\eta)\mu$ for some
$f:\mathfrak{g}_\mu^\perp \to G$.
Then $\varphi$
preserves the $-$KKS symplectic form
if and only if
\begin{align}\label{E:preserveKKS}
\left\langle\mu, \left[\zeta_1, \zeta_2\right]\right\rangle
&= \left\langle f(\eta)\mu, 
\left[\left(Df(\eta)\cdot \zeta_1\right) f(\eta)^{-1},
\left(Df(\eta)\cdot \zeta_2\right) f(\eta)^{-1} 
\right]
\right\rangle \\
&= \left\langle \mu, 
\left[f(\eta)^{-1} \left(Df(\eta)\cdot \zeta_1\right),
f(\eta)^{-1} \left(Df(\eta)\cdot \zeta_2\right)
\right]
\right\rangle \notag
\end{align}
for all $\eta,\zeta_1,\zeta_2 \in \mathfrak{g}_\mu^\perp$.
\end{lemma}

This was the inspiration that led to the constructive slice theorem in the next section.

\subsection{A constructive slice theorem for $T^*SO(3)$}\label{ssect:sliceSO3}

For the reason outlined above, we consider
maps $\varphi:T_\mu\mathcal{O}_\mu\to \mathcal{O}_\mu$ 
such that $\varphi(\0) = \mu$ and $D\varphi(\0)$ is the identity,
that preserve the $-$KKS form.
The KKS forms for $SO(3)$, for any $\mu$, are $\frac{1}{\|\mu\|}$ 
times the signed area form on
$\mathcal{O}_\mu \cong S^2$,
with the sign corresponding to the outward-pointing normal for the $+$KKS form,
and the inward-pointing normal for the $-$KKS form.
Thus a map $\varphi:T_\mu\mathcal{O}_\mu\to \mathcal{O}_\mu$ that preserves the $\pm$KKS form
is just an area-preserving map from $\R^2$ to $S^2(\|\mu\|)$,
where $S^2(\|\mu\|)$ is the sphere of radius $\|\mu\|$ centred at the origin.

Without loss of generality, we consider $\mu = (0,0,\mu_z)$, with $\mu_z > 0$.
Consider the usual polar coordinates $(r,\theta)$ on the plane and spherical coordinates $(\theta, \phi)$
on the unit sphere, where $\theta$ is usual angle coordinate in the $xy$-plane, and $\phi$ is the angle from the positive $z$ axis.
Note that the signed area
$d\theta \wedge \d\phi$ is the $-$KKS form.
We seek an area-preserving map $\varphi:T_\mu\mathcal{O}_\mu\to \mathcal{O}_\mu$,
such that $\varphi(\0) = \mu$ and $D\varphi(\0) = Id$,
and require also that $\varphi$ be equivariant with respect to $G_\mu$, 
which consists of rotations around the $z$ axis.
We make an Ansatz that $\varphi$ preserves $\theta$.
It can be shown that the unique $\varphi$ satisfying all of these requirements
is given by
\[
\phi = 2 \arcsin \left(\frac{r}{2\|\mu\|}\right).
\]
To write this in the form of Lemma \ref{preserveKKS},
$\varphi\left(-\ad_\eta^* \mu\right) = \Ad^*_{f(\eta)^{-1}}\mu$, we define
\[
f(\eta) := \exp \left( 2 \arcsin \left(\frac{\| \eta \|}{2}\right) \frac{\eta}{\|\eta\|}\right).
\]
where $\exp$ be the usual matrix exponential.

\smallskip

Comparing  \eqref{E:slicecond} and \eqref{E:preserveKKS2} in the Appendix, 
we may guess that a factor involving $\|\mu\| / \| \mu + \nu\|$ should be inserted
in order to produce a symplectic tube.
The solution may be discovered by trial and error, 
however we will proceed systematically from the Ansatz
\begin{equation}\label{E:Ansatz}
F(\nu, \eta) = \exp\left(h(\nu,\eta) \,  \frac{\eta}{\|\eta\|}\right),
\end{equation}
for some real-valued $h$.
Note that the term $F(\nu, \eta)^{-1} \left(DF(\nu, \eta)\cdot (\dot \nu, \zeta)\right)$ that appears in 
the Tube Condition in Proposition \ref{tubecond} takes the following form when $\dot \nu = 0$,
\begin{align*}
F(\nu, \eta)^{-1} \left(DF(\nu, \eta)\cdot (0, \zeta)\right)
= \exp \left(- \hat \v\right) \left.\frac{d}{dt} \right|_{t=0} \exp \left(\hat \v + t \hat \w\right),
\end{align*}
where $\hat \v := h(\nu,\eta) \frac{\eta}{\|\eta\|}$
and $\hat \w := h(\nu,\eta) \frac{\zeta}{\|\eta\|}$, 
and the hat map $\v \mapsto \hat \v$ is defined by 
\[
\hat \v = \begin{pmatrix}
0 & -v_3 & v_2 \\
v_3 & 0 & -v_1 \\
-v_2 & v_1 & 0 
\end{pmatrix}.
\]
We compute this quantity with the aid of \textbf{Rodrigues' rotation formula} (see \cite{MR99}):
\begin{align*}
\exp (\hat \v) 
&= I + \frac{\sin \|\v\|}{\|\v\|} \hat \v+ \frac{1 - \cos \|\v\|}{\|\v\|^2} \hat \v^2 
= I + \frac{\sin \|\v\|}{\|\v\|} \hat \v+ \frac{2 \sin^2 \frac{\|\v\|}{2}}{\|\v\|^2} \hat \v^2.
\end{align*}

\begin{lemma}
For general orthogonal $\v$ and $\w$,
\begin{align} \label{E:Dexpperp}
\exp \left(- \hat \v\right) \left.\frac{d}{dt} \right|_{t=0} \exp \left(\hat \v + t \hat \w\right)
& = \frac{\sin \|\v\|}{\|\v\|}
 \, \hat \w - \frac{2\sin^2 \frac{\|\v\|}{2}}{\|\v\|^2} \, \left(\v \times \w\right)\hat{} \,.
\end{align}
\end{lemma}

\begin{proof}
By the naturality property of $\exp$, and the fact that 
$R\hat \w R^{-1} = \left(R \w\right)^{\hat{}}$, it suffices to prove the claim for 
$\v = (v_x,0,0)$ and $\w = (0,w_y,0)$.
This is a straightforward calculation.
\end{proof}

\begin{lemma}\label{zetaperp}
If $\eta, \zeta\in \lieg_\mu^\perp$ and $\zeta$ is perpendicular to $\eta$,
then
\begin{align*}
F(\nu, \eta)^{-1} \left(DF(\nu, \eta)\cdot (\dot \nu, \zeta)\right)
& = 
\frac{\sin h}{\|\eta\|} \, \hat \zeta
 - \frac{2\sin^2 \frac{h}{2}}{\|\eta\|^2} \, \left(\eta \times \zeta\right)\hat{} \,,
\end{align*}
and $\eta \times \zeta \in \lieg_\mu$.
\end{lemma}

\begin{proof}
$\displaystyle
\left.\frac{d}{dt} \right|_{t=0} \|\eta + t \zeta\| = 0
$ and
$\displaystyle
\left.\frac{d}{dt} \right|_{t=0} \frac{\eta + t\zeta}{\|\eta + t \zeta\|} 
= \frac{\zeta}{\|\eta\|}
$.
Then 
\[
DF(\nu, \eta)\cdot (0, \zeta)
= 
\left.\frac{d}{dt} \right|_{t=0} 
\exp \left(
h(\nu, \eta)\, \frac{\eta}{\|\eta\|}
+ t \, h(\nu, \eta)\, \frac{\zeta}{\|\eta\|}
\right).
\]

From \eqref{E:Dexpperp}, 
with $\hat \v = h(\nu,\eta) \frac{\eta}{\|\eta\|}$
and $\hat \w = h(\nu,\eta) \frac{\zeta}{\|\eta\|}$,
\begin{align*}
F(\nu, \eta)^{-1} \left(DF(\nu, \eta)\cdot (\dot \nu, \zeta)\right)
&=
\exp \left(- \hat \v\right) \left.\frac{d}{dt} \right|_{t=0} \exp \left(\hat \v + t \hat \w\right)\\
& = \frac{\sin \|\v\|}{\|\v\|}
 \, \hat \w - \frac{2\sin^2 \frac{\|\v\|}{2}}{\|\v\|^2} \, \left(\v \times \w\right)\hat{} \\
& = \frac{\sin h}{h}
 \, \hat \w - \frac{2\sin^2 \frac{h}{2}}{h^2} \, \left(\v \times \w\right)\hat{} \\
& = 
\frac{\sin h}{\|\eta\|} \, \hat \zeta
 - \frac{2\sin^2 \frac{h}{2}}{\|\eta\|^2} \, \left(\eta \times \zeta\right)\hat{} \,.
\end{align*}
Since $\eta,\zeta\in \lieg_\mu^\perp$ and $\eta \perp \zeta$,
it follows that
$\eta \times \zeta \in \lieg_\mu$.
\end{proof}

\medskip
\noindent
We now calculate the Tube Condition in Proposition \ref{tubecond} under the Ansatz \eqref{E:Ansatz}.
The following lemma covers the case of $\zeta_1,\zeta_2$ both parallel to $\eta$,
which includes the case of $\zeta_1 = \zeta_2 = 0$ (the ``$\dot \nu - \dot \nu$'' case).
Though motivated by our study of the $SO(3)$ case, the following lemma
applies to general $G$. It is proven in the Appendix.

\begin{lemma}\label{lem:ansatz}
Suppose $\displaystyle F(\nu, \eta) = \exp\left(h(\nu,\eta) \,  \frac{\eta}{\|\eta\|}\right)$,
for some $h: \lieg_{\mu}^* \times \lieg_\mu^\perp \to \R$. 
Then the Tube Condition in Proposition \ref{tubecond} is automatically satisfied
(regardless of the definition of $h$) for all 
$\left(\dot \nu_1, \zeta_1\right), \left(\dot \nu_2, \zeta_2\right)$ such that
$\zeta_1$ and $\zeta_2$ are parallel to $\eta$.
\end{lemma}

From this and the bilinearity of Condition \eqref{E:slicecond} in the Appendix, we are left with three cases to check.

\medskip
\noindent
\textbf{Case: 
$\zeta_1$ and $\zeta_2$ both perpendicular to $\eta$}.
In this case, $\zeta_1$ and $\zeta_2$ are parallel to each other, and
\begin{align*}
&\left\langle 
\mu + \nu, 
\left[ 
F(\nu, \eta)^{-1} \left(DF(\nu, \eta) \cdot \left(0, \zeta_1\right)\right),
F(\nu, \eta)^{-1} \left(DF(\nu, \eta) \cdot \left(0, \zeta_2\right)\right)
\right]
\right\rangle \\
&= \left\langle 
\mu + \nu, 
\left(\frac{\sin h}{\|\eta\|} \, \zeta_2\right)
\times \left(\frac{2\sin^2 \frac{h}{2}}{\|\eta\|^2} \, \left(\eta \times \zeta_1\right) \right)
-
\left(\frac{\sin h}{\|\eta\|} \, \zeta_1 \right)
\times \left(\frac{2\sin^2 \frac{h}{2}}{\|\eta\|^2} \, \left(\eta \times \zeta_2\right)\right)
\right\rangle \\
&= 0 = \left\langle \mu, \left[\zeta_1, \zeta_2\right]\right\rangle.
\end{align*}
Therefore the Tube Condition in Proposition \ref{tubecond} is satisfied, for any $h$.

\medskip
\noindent
\textbf{Case: $\dot \nu_1 = \dot \nu_2 = 0$,
$\zeta_1$ is parallel to $\eta$ and
$\zeta_2$ is perpendicular to $\eta$}.\\
By Lemma \ref{zetaperp},
\begin{align*}
&\left\langle 
\mu + \nu, 
\left[ 
F(\nu, \eta)^{-1} \left(DF(\nu, \eta) \cdot \left(0, \zeta_1\right)\right),
F(\nu, \eta)^{-1} \left(DF(\nu, \eta) \cdot \left(0, \zeta_2\right)\right)
\right]
\right\rangle \\
&= \left\langle 
\mu + \nu, 
\left(\frac{\pa h}{\pa \eta} (\nu, \eta) \cdot \zeta_1\right) \frac{\eta}{\|\eta\|}
\times
\left(
\frac{\sin h}{\|\eta\|} \, \zeta_2
 - \frac{2\sin^2 \frac{h}{2}}{\|\eta\|^2} \, \left(\eta \times \zeta_2\right)
\right)
\right\rangle \\
&= \left\langle 
\mu + \nu, 
\left(\frac{\pa h}{\pa \eta} (\nu, \eta) \cdot \zeta_1\right) \frac{\eta}{\|\eta\|}
\times
\left(
\frac{\sin h}{\|\eta\|}
\right)  \, \zeta_2
\right\rangle\\
&=
\pm \|\mu + \nu\| \left(\frac{\pa h}{\pa \eta} (\nu, \eta) \cdot \zeta_1\right)
\left(\frac{\sin h}{\|\eta\|}\right)
\|\zeta_2\|,
\end{align*}
where the sign is the sign of $\mu \cdot \left(\eta \times \zeta_2\right)$.

For the Tube Condition in Proposition \ref{tubecond} to be satisfied,
this must equal
$\mu \cdot \zeta_1 \times \zeta_2
= \pm \| \mu \| \| \zeta_1\| \| \zeta_2\|$
for all $\zeta_1,\zeta_2$, 
which occurs if and only if
\begin{align}\label{E:etaparperp}
\mathrm{sgn}(\zeta_1 \cdot \eta)\|\mu + \nu\| \left(\frac{\pa h}{\pa \eta} (\nu, \eta) \cdot \zeta_1\right)
\left(\frac{\sin h}{\|\eta\|}\right)
= \| \mu \| \| \zeta_1\| .
\end{align}
If we further assume that $h$ depends on $\eta$ only through $\|\eta\|$, then \eqref{E:etaparperp}
becomes:
\begin{align}\label{E:etaparperp2}
\|\mu + \nu\| \frac{\pa h}{\pa \|\eta\|} (\nu, \|\eta\|)
\left(\frac{\sin h}{\|\eta\|}\right)
= \| \mu \|.
\end{align}

\medskip
\noindent
\textbf{Case: $\zeta_1 = 0$ and $\zeta_2$ is perpendicular to $\eta$}.
\begin{align*}
&\left\langle 
\mu + \nu, 
\left[ 
F(\nu, \eta)^{-1} \left(DF(\nu, \eta) \cdot \left(\dot \nu_1, 0\right)\right),
F(\nu, \eta)^{-1} \left(DF(\nu, \eta) \cdot \left(0, \zeta_2\right)\right)
\right]
\right\rangle \\
&+
\left\langle
\dot \nu_2, F(\nu, \eta)^{-1} \left(DF(\nu, \eta) \cdot \left(\dot \nu_1, 0\right)\right)
\right\rangle
-
\left\langle
\dot \nu_1, F(\nu, \eta)^{-1} \left(DF(\nu, \eta) \cdot \left(0, \zeta_2\right)\right)
\right\rangle \\
&=
\left\langle 
\mu + \nu, 
\left( 
\left(\frac{\pa h}{\pa \nu} (\nu, \eta) \, \dot \nu_1\right) \frac{\eta}{\|\eta\|}\right)
\times
\left(\frac{\sin h}{\|\eta\|} \, \zeta_2
\right)
\right\rangle \\
&
+
\left\langle
\dot \nu_1, 
\frac{2\sin^2 \frac{h}{2}}{\|\eta\|^2} \, \left(\eta \times \zeta_2\right)
\right\rangle \\
&=
\pm \left[ \|\mu + \nu\|
\left(\frac{\pa h}{\pa \nu} (\nu, \eta) \, \dot \nu_1 \right) 
\left(\sin h\right) \, \frac{\|\zeta_2\|}{\| \eta\|}
+
2 \dot \nu_1 \, 
\left(\sin^2 \frac{h}{2} \right)
\frac{\|\zeta_2\|}{\|\eta\|} \right],
\end{align*}
where the sign is the sign of $\mu \cdot \left(\eta \times \zeta_2\right)$.
  
For the Tube Condition in Proposition \ref{tubecond} to be satisfied, this expression must
equal zero, for all $\dot \nu_1$. If $h(\nu, \eta)\ne 0$, a factor of $\sin (h/2)$ cancels,
giving the equivalent condition
\begin{align}\label{E:nuetaperp}
\|\mu + \nu\| \,
\frac{\pa h}{\pa \nu} \, 
\cos \frac{h}{2}
+
\sin \frac{h}{2}
&=0,
\end{align}

\begin{theorem}\label{SO3tube}
Let
\begin{align*}
\Phi(g, \nu, \ad\nolimits_\eta^* \mu) 
= \left(
g F(\nu, \eta)^{-1}, 
\Ad\nolimits_{F(\nu, \eta)^{-1}}^*
\left(\mu + \nu\right)
\right),
\end{align*}
where
\[
F(\nu, \eta) = \exp \left(
2  \arcsin \left(\frac{1}{2} \| \eta \| \sqrt{\frac{\| \mu \|}{\| \mu + \nu\|}} \right)
\frac{\eta}{\| \eta \|}
\right).
\]
Then $\Phi^* \Omega_c = \Omega_Y$.
The domain of definition of $\Phi$ is $SO(3) \times 
\left(U \subset \mathfrak{so}(3)_\mu^* \times T_\mu \mathcal{O}_\mu\right)$,
where 
\[
U = \left\{ \left(\nu, \ad\nolimits_\eta^* \mu\right) : \nu > -\|\mu\| \textrm{ and } 
\| \eta \| < 2 \sqrt{\frac{\| \mu +\nu \|}{\| \mu\|}} \right\}.
\]

\end{theorem}

\begin{proof}
Let $x =  \| \eta \| \sqrt{\frac{\| \mu \|}{\| \mu + \nu\|}}$, and
\[
h(\nu, \|\eta\|) = 2 \arcsin \frac{x}{2} =
2  \arcsin \left(\frac{1}{2} \| \eta \| \sqrt{\frac{\| \mu \|}{\| \mu + \nu\|}} \right).
\]
Then $dh/dx =2 / \sqrt{4 - x^2} = 1/\cos(h/2)$, so
\begin{align*}
\frac{\partial h}{\partial \nu} \cos \frac{h}{2} = \frac{- \|\eta\|}{2\|\mu + \nu\|} 
\sqrt{\frac{\|\mu\|}{\| \mu + \nu \|}}, 
\end{align*}
which implies that \eqref{E:nuetaperp} is satisfied. Also,
\begin{align*}
\frac{\partial h}{\partial \|\eta\|} \sin h
= 2 \sin \frac{h}{2} \sqrt{\frac{\| \mu\|}{\| \mu + \nu \|}}
= \frac{\| \eta \| \| \mu\|}{\| \mu + \nu \|},
\end{align*}
so \eqref{E:etaparperp2} is satisfied.
\end{proof}

\begin{remark}
The restriction of $\Phi$ to a level set defined by $(R,\nu)= (Id, \nu_0)$ has as its image 
an open neighbourhood of $\mu + \nu_0$ in the coadjoint orbit $\mathcal{O}_{\mu + \nu_0}$,
which is a sphere and is isomorphic to the symplectic reduced space at $\mu + \nu_0$.
For any choice of $\nu_0$, the neighbourhood covers almost the entire sphere, excluding only the antipodal
point $-(\mu + \nu_0)$.
\end{remark}

\begin{remark}
This $\Phi$ has a limited uniqueness property.
From the Tube Condition in Proposition \ref{tubecond}, any symplectic tube must of be expressed in 
terms of an $F$ as stated in the theorem. Any $F$ can be expressed as the exponential
of some function $\lieg_\mu^* \times \lieg_\mu^\perp \mapsto \lieg$.
If that function is of the form $h(\nu, \| \eta\|) \, \eta / \| \eta \|$, then
the two conditions \eqref{E:etaparperp2} and \eqref{E:nuetaperp} are sufficient to determine $h(\nu, \| \eta \|)$.
\end{remark}

\subsection{Actions of $SO(3)$ on arbitrary configuration spaces} \label{ssect:arbitrary}

The results of the previous section can be used
to construct symplectic slices for any free and proper cotangent-lifted action of $SO(3)$
on $T^*Q$, for arbitrary $Q$. 

\begin{proposition}\label{prop_gen_tube}
Suppose $SO(3)$ acts freely on a manifold $Q$, and by cotangent lifts on $T^*Q$.
Let \begin{align*}
\tau:SO(3)\times (S \subset N)  &  \rightarrow Q\\
\left(  R,s\right)   &  \longmapsto R \exp_{q_0}s
\end{align*}
be the tube given by Theorem \ref{Palais} (Palais' Slice Theorem).
Let $\Phi: SO(3) \times \mathfrak{so(3)}_\mu^* \times T_\mu \mathcal{O}_{\mu_0} \to 
SO(3) \times \mathfrak{so}(3)^*$
be defined as in Theorem \ref{SO3tube}.
Then the following composition 
\begin{equation}\label{E:SO3tube}
SO(3) \times \mathfrak{so(3)}_\mu^* \times T_\mu \mathcal{O}_\mu \times T^*S
\overset{\left(\Phi, \mathrm{id}\right)}{\rightarrow}
SO(3) \times  \mathfrak{so(3)}^* \times  T^*S
\cong T^*(SO(3) \times S) \overset{T^*\tau^{-1}}{\rightarrow} T^*Q
\end{equation}
(where the central isomorphism is left-trivialisation)
is an $SO(3)$-equivariant symplectomorphism
with respect to the canonical symplectic form on $T^*Q$ and the symplectic form $\Omega_Y$
defined in \eqref{E:OmegaY}.
\end{proposition}

In the case $Q = \R^n$, we have $\tau(R, s) = R(q_0 + s)$.
We now explain the cotangent lift $T^*\tau^{-1}$ that appears in \eqref{E:SO3tube}.
Writing $q := R(q_0 + s)$,
the tangent space $T_q Q$ splits into the direct sum of two subspaces:  
\begin{align*}
\mathfrak{so}(3) q &:= \{ \xi  q : \xi \in \mathfrak{so}(3) \}
& \qquad \textrm{(``group direction'', tangent to $Gq$),} \\
R N &:= \{ R v : v \in N \} &\textrm{(``slice direction'').}
\end{align*}
The cotangent space $T_q^*Q$ has a corresponding splitting into
$\mathfrak{so}(3)^*$ (group direction) and $RN^*$ (slice direction); note that $(RN)^* = R(N^*)$.
The tangent lift of $\tau$ is given by 
\begin{equation}\label{E:tanlift}
(\xi,\dot s) \in T_{(R,s)} (SO(3) \times S) \mapsto R\left(\xi (q_0 + s) + \dot s\right) \in T_{q} Q,
\end{equation}
where $R\left(\xi (q + s) \right)$ is in the group direction and $R\dot s$
is in the slice direction,
and we have used left-trivialisation to write $(R, \xi) \in SO(3) \times \mathfrak{so}(3) \cong T(SO(3)$.
The cotangent lift $T^*\tau^{-1}$,
also has two components, in the group and slice directions:
\begin{align*}
(\mu, \sigma) \in T_{(R,s)}^*(SO(3) \times S) \mapsto (\alpha_\mu(q) + R\sigma) \in T^*_{q} Q,
\end{align*} 
where $\alpha_\mu(q) \in(\mathfrak{so}(3) q)^*$ and $R\sigma \in RN^*$.
To define these components explicitly,
we pair them with the components of a general tangent vector.
Since $\tau$ is a diffeomorphism, all such tangent vectors can be expressed in the form 
\eqref{E:tanlift}. We have
\begin{align*}
\left\langle \alpha_\mu(q) + R\sigma, R\left(\xi (q + s) + \dot s\right) \right\rangle
&:= \left\langle \mu, \xi\right\rangle + \left\langle \sigma, \dot s\right\rangle,
\end{align*}
i.e.
\begin{align*}
\left\langle \alpha_\mu(q), R \left(\xi (q + s) \right) \right\rangle &:= \left\langle \mu, \xi\right\rangle, \quad \textrm{and} \\
\left\langle R\sigma, R\dot s\right \rangle &:= \left\langle \sigma, \dot s\right\rangle.
\end{align*}

A similar strategy allows one to construct symplectic slices for some \textit{non-free}
actions of $SO(3)$ on general cotangent bundles. However we leave this topic for a later paper.

\section{Dynamics in slice coordinates} \label{Normal_forms_SO(3)}


\subsection{Dynamics on $T^*SO(3)$}

In this section we describe the motion on $T^*SO(3)$ in normal form coordinates near a fixed non-zero momentum $ (0,0,\mu_0)\in \mathbb{R}^3 \simeq so(3)^*$. 
We identify $T_{\mu_0}\mathcal{O}_{\mu_0}$ with $so(3)_{\mu_0}^\perp$ via
$\ad_\eta^*\mu_0 \mapsto \eta$.
Then the $SO(3)$-equivariant symplectomorphism given by Theorem \ref{SO3tube} takes the form 
\begin{align*}
\Phi : SO(3) \times so(3)^*_{\mu_0} \times so(3)_{\mu_0}^\perp &\to SO(3) \times so(3)^* \\
\Phi( R, \nu, \eta ) &=  \left( R \left(F(\nu,\eta ) \right)^{-1}, 
\,F(\nu, \eta) (\mu_0 +\nu)   \right)
\end{align*}
with
\begin{equation}
\label{F_formula}
F(\nu, \eta) = \exp \left(\theta \,\frac{\hat \eta}{\|\eta\|} \right)
\quad 
 \text{where}  \quad \sin \frac{\theta}{2} =  \frac{1}{2} \|\eta\| \sqrt{\frac{\|\mu_0\|}{\|\mu_0 + \nu  \|} }\,.\end{equation}
 Specifically, one has the change of variables
 \begin{align*}
 SO(3) \times so(3)^*_{\mu_0} \times so(3)_{\mu_0}^\perp &\to  SO(3) \times so(3)^* \\
( R, \nu, \eta ) &\to  \Phi( R, \nu, \eta ):= (S, \mu )
  \end{align*}
where
\begin{align}
\mu_1&=\, \eta_y \sqrt{\mu_0(\mu_0 + \nu) \left(1-\frac{\mu_0}{4(\mu_0 + \nu)}(\eta_x^2 +\eta_y^2)  \right)} \label{mu_1}\\
\, \nonumber \\
\mu_2&= - \eta_x \sqrt{\mu_0(\mu_0 + \nu) \left(1-\frac{\mu_0}{4(\mu_0 + \nu)}(\eta_x^2 +\eta_y^2) \right) }  \label{mu_2} \\
\, \nonumber \\
\mu_3 &= (\mu_0+\nu) -\frac{1}{2}\mu_0(\eta_x^2 +\eta_y^2)  \label{mu_3}
\end{align}
and $S = R F(\nu, \eta)^{-1}.$
The  symplectic form on $SO(3) \times so(3)_{\mu_0}^* \times so(3)_{\mu_0}^\perp$ is given by
%
%
%
\begin{equation*}
\Omega_Y(R, \nu, \eta) = 
\left[
\begin{array}{ccc}
(\mu_0 + \nu) \mathbb{J} &0&0\\
0 & \mathbb{J} &0 \\
0&0& -\mu_0 \mathbb{J}
\end{array}
\right] 
\end{equation*}
where we use the notation $\displaystyle{\mathbb{J} :=\left[ \begin{array}{cc} 0&1\\1&0\end{array} \right]}\,.$
Note that this matrix does not depend on $\eta$. In $(R, \nu, \eta)$ coordinates the (spatial) momentum map  $J(S, \mu) =  S \mu$ reads:
\[J\left(R ,\nu, \eta \right) =   R \,(\mu_0+\nu\,).\]
%
It is useful to recall  that in slice coordinates the Marsden-Weinstein reduced spaces at $\left(R ,\nu, \eta \right)$, which are $J^{-1}\left(R(\mu_0 +\nu ) \right)/SO(3)_{R(\mu_0+\nu)}^*$, are
all isomorphic to the linear space $T_{\mu_0} {\mathcal O}_{\mu_0}\cong so(3)_{\mu_0}^\perp.$ 
The symplectic leaves ${\mathcal O}_{\mu_0+\nu} = S^2(\|\mu_0+\nu\|)$ of $so(3)^*$ are modelled (locally) as canonical linear spaces, and ``indexed" by $\nu.$

\medskip
Consider now a   Hamiltonian $\tilde H( S, \mu)$ on $SO(3) \times so(3)^*$. Applying the change of coordinates  given by $\Phi,$  we have: $H  (R, \nu, \eta) : =( \tilde H \circ \Phi)( R, \nu, \eta)$
and the equations of  motion become
\begin{align}\label{E:eqsmotion}
\left[
\begin{array}{cccccc}
\xi_x\\
\xi_y\\
\,\\
\xi_z\\
\dot \nu\\
\,\\
\dot \eta_x\\
\dot \eta_y
\end{array}
\right] = 
\left[
\begin{array}{cccccccc}
0 & \frac{1}{\mu_0+\nu} & \, & 0 & 0  &\,  & 0 & 0\\
-\frac{1}{\mu_0+\nu} & 0  & \,& 0 & 0 &\, & 0 & 0\\
\,\\
0 & 0 & \,  & \,\,0 & 1 & \, &0&0\\
0 & 0 &\,& -1 & 0 &\,  &0&0\\
\,\\
0 & 0 &\,& 0 & 0 & \, & 0&- \frac{1}{\mu_0}\\
0 & 0 & \,& 0 & 0 \,& &\frac{1}{\mu_0}&0
\end{array}
\right] 
\left[
\begin{array}{cccccc}
\left( R^{-1}\partial_R H \right)_x\\
\left( R^{-1}\partial_R H \right)_y\\
\,\\
\left( R^{-1}\partial_R H \right)_z\\
\partial_\nu H\\
\,\\
\partial_{\eta_x} H\\
\partial_{\eta_y} H
\end{array}
\right]
\end{align}
where $\xi= R^{-1}\dot R \in so(3).$
In particular, if $H$ is $SO(3)$-invariant and  $h(\nu, \eta) := H(\cdot , \nu, \eta) $, we have
\begin{align}
\dot \nu &=0 \label{nu_const} \\
\dot \eta &= -\frac{1}{\mu_0} \mathbb{J} \,\nabla_\eta h
\label{zeta_eq}
\end{align}
with reconstruction equations:
\begin{align}\label{can_recons_1}
R(t)^{-1} \dot R(t) = \xi(t)  = \,\left( 
\begin{array}{ccc}
0\\
0\\
\frac{\partial h}{\partial \nu}\Big|_{(\nu(t), \eta(t))}
\end{array}
\right).
\end{align}

Note that this reconstructs $R(t)$, not the body's attitude $S(t)$.
At a given time $t_1$, once $\nu(t_1), \eta(t_1)$ and $R(t_1)$ have been calculated by integrating
\eqref{nu_const}, \eqref{zeta_eq} and \eqref{can_recons_1}, the attitude can be computed simply as 
\[
S(t_1) = R(t_1) F(\nu(t_1), \eta(t_1))^{-1}.
\]

\vspace{0.3cm}
A \textit{relative equilibrium} is a steady motion in a group direction. 
In the original left-trivialised coordinates $(S,\mu)$, a relative equilibrium with velocity $\xi_0$ is
a trajectory of 
the form $S(t) = \exp(t\xi_0) S_0$ with $\mu$ constant.
In slice coordinates, $\mu$ constant is equivalent to $\nu$ and $\eta$ constant,
and in this case
$S(t) = R(t) F^{-1}(\nu_0, \eta_0)$ implies $R(t) = \exp(t\xi_0) R_0$.
Thus in slice coordinates, a relative equilibrium with velocity $\xi_0$ is a trajectory 
in which $\eta$ is an equilibrium of \eqref{zeta_eq}
and the velocity $\xi(t)$ given by \eqref{can_recons_1} has the constant value $\xi_0$.

Since $\nu= const.= \nu_0$, the reduced Hamiltonian depends dynamically on $\eta$ only, whereas $\nu_0$ affects the motion as  an external parameter.
Thus $h(\eta; \nu_0)$ is a one degree of freedom canonical system on a symplectic vector space. 
The phase curves for \eqref{zeta_eq}  fill in the
$so(3)_{\mu_0}^{\perp}$-phase plane as level sets of the energy integral $h(\eta; \nu_0)= const.$ In particular,  any  $SO(3)$-invariant system on $T^*SO(3)$ is integrable. 

  Note that the reconstruction equation \eqref{can_recons_1} reduces to reconstruction on the Abelian group $SO(3)_{\mu_0}$ and it leads to rotations about the $z$-axis. Specifically, if $\eta(t)$ is a solution for \eqref{zeta_eq}, then
  \begin{equation}
  \label{recon_Ab}
  R(t) =  
  \left[
  \begin{array}{ccc}
  \cos \theta(t) & -\sin \theta(t) &0\\
  \sin\theta(t) & \,\,\,\,\,\,\cos\theta(t) & 0\\
  0&0&1
  \end{array}
  \right]
  \end{equation}
where
  \begin{align*}
  \label{recon_eta}
\theta(t) =  \frac{\partial h}{\partial \nu}\Big|_{\left(\nu_0, \eta(t)\right)}\,.\end{align*}

\subsection{The Euler-Poinsot rigid body}\label{ssect:freerigid}

The Hamiltonian of the Euler-Poinsot (free) rigid body is (see, for instance,  \cite{MR99}):
\[h(\mu_1, \mu_2, \mu_3)= \frac{1}{2} \left(
\frac{\mu_1}{\mathbb{I}_1}+ \frac{\mu_2}{\mathbb{I}_2}+\frac{\mu_3}{\mathbb{I}_3}
\right)
\]
where $\mathbb{I}_i$ are the principal moments of inertia. 
 Using the formulae \eqref{mu_1}--\eqref{mu_3}  the Hamiltonian $h$ reads:

\begin{align}
h(\nu, \eta):= \frac{1}{2} (\mu_0+\nu)
&\left[
\mu_0 \left( 1- \frac{\mu_0}{4(\mu_0 + \nu)} \left(\eta_x^2+\eta_y^2  \right) \right) \left( \frac{\eta_x^2}{\mathbb{I}_2}+ \frac{\eta_y^2} {\mathbb{I}_1}  \right)
 \right. \nonumber\\
&\,\,\, \quad \left. +   \left( 1- \frac{\mu_0}{2(\mu_0 + \nu)} \left(\eta_x^2+\eta_y^2  \right) \right)^2\frac{(\mu_0+\nu)}{\mathbb{I}_3}
    \right]
\end{align}
One may deduce easily the  stability criteria, as well as sketch the Marsden-Weinstein reduced phase-space at any momentum $(\mu_0+\nu).$
The super-integrability of the Euler-Poinsot rigid body (that is  the case when $\mathbb{I}_1=\mathbb{I}_2$)  is transparent, as $h$ becomes a function of $|\eta|^2$ only.

\begin{figure}
    \mbox{
      { 
           \includegraphics[scale=0.41] {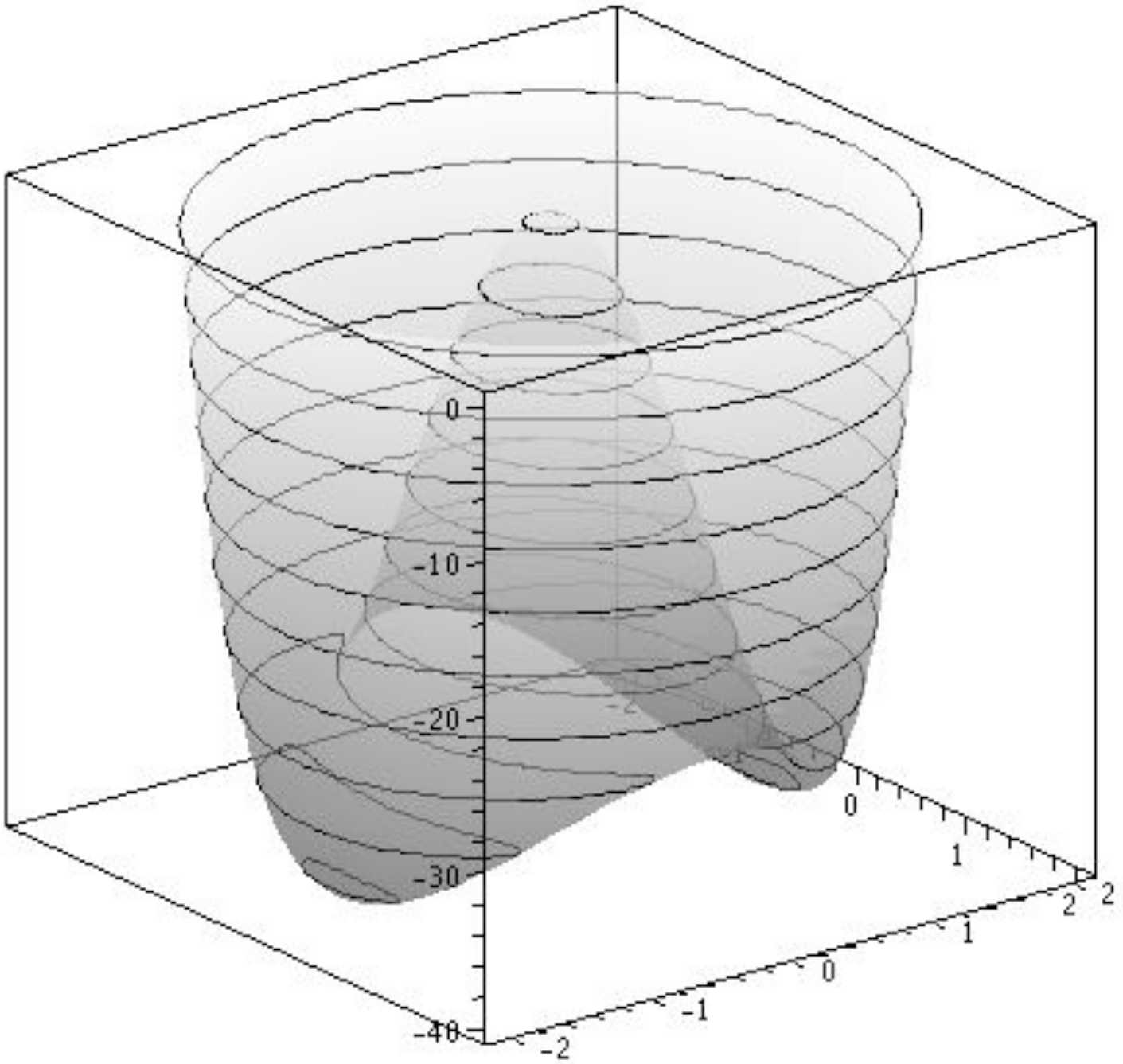} 
           }    \quad \quad \quad 
      { 
           \includegraphics[scale=0.41] {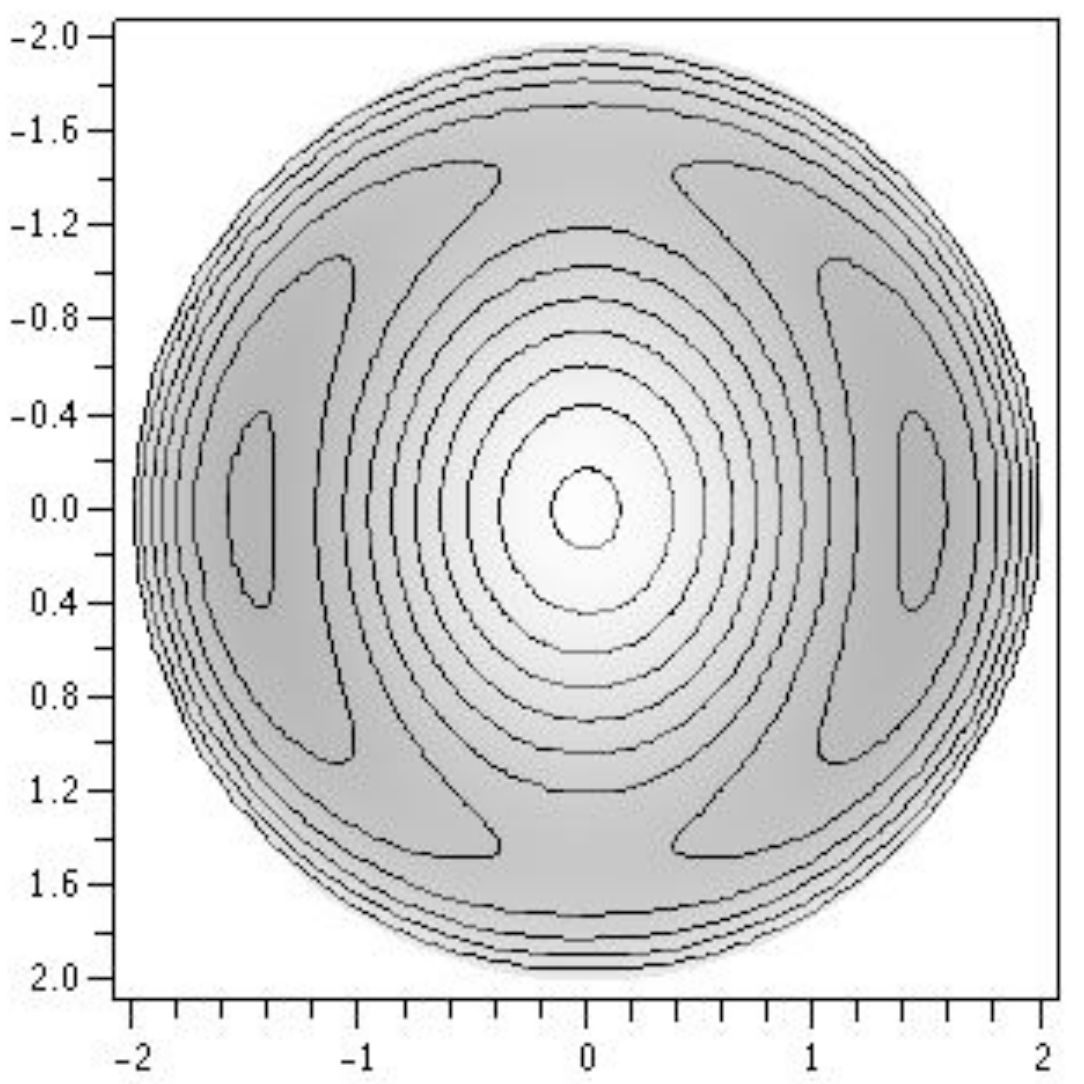} 
           } 
           }
              \caption{The Marsden-Weinstein reduced phase-space for the free rigid body in slice coordinates. By \eqref{zeta_eq}, we have $\nu=\nu_0= const.$ The phase curves are retrieved as the level sets of the Hamiltonian $h(\eta_x, \eta_y; \nu_0)$. Left: 3-d view. Right: top view.}
\end{figure}

Is is known that  the rigid body  accepts a canonical action-angle description as given by Serret-Deprit-Andoyer coordinates. A comprehensive description of these coordinates and their generalisation to \textit{regularised} coordinate charts  which cover the co-adjoint sphere minus the antipodal point of the relative equilibrium $(0, 0, \mu_0)$ can be found in \cite{Fasso_4} (Section 4) and the references therein. A direct comparison of the slice  and the \textit{regularised} Serret-Deprit-Andoyer coordinates shows that they provide \textit{identical} parametrisations of $T_{\mu_0}{\mathcal{O}}_{\mu_0} \equiv so(3)_{\mu_0}^\perp.$  
Specifically, $(\eta_x, \eta_y)$ are in fact regularised Serret-Andoyer-Deprit
coordinates.
The slice parameterisation uses a global attitude description R(t). If an explicit local coordinate system
is sought, Serret-Deprit-Andoyer is probably the best choice,
because they are action-angle coordinates with a very simple relation to Euler angles.  
The relationship between the two parametrisations will be discussed in detail in future work.

\subsection{Cotangent-bundle rotationally invariant systems}\label{ssect:CBrot}

Consider a $SO(3)$-invariant Hamiltonian system $H: T^*Q \to \mathbb{R}$ and let $(q_0, p_0) \in T^*Q$ be a point on a relative equilibrium with group velocity $\xi$ and momentum $\mu_0.$   We apply now Proposition \ref{prop_gen_tube} where $(q_0, p_0)$ is  the base point of the Palais tube.
It follows that in a neighbourhood  of $(q_0, p_0)$ the phase space is symplectomorphic to $SO(3) \times so(3)_{\mu_0}^* \times so(3)_{\mu_0}^\perp \times T^*S$, where we identified $T_{\mu_0}{\mathcal{O}}_{\mu_0} \equiv so(3)_{\mu_0}^\perp.$ Since the Hamiltonian is $SO(3)$ invariant, in slice coordinates 
\begin{equation}
\label{slice_coord}
(R, \nu, \eta, s, \sigma ) \in SO(3) \times so(3)_{\mu_0}^* \times so(3)_{\mu_0}^\perp \times T^*S
\end{equation}
 it can be written as $h=h(\nu, \eta, s, \sigma).$
The equations of motion take the form
\begin{align}
 \dot \nu&= 0  \label{can_system_1}\\
 \dot \eta &= -\frac{1}{\mu_0} \,  \mathbb{J} \, \partial _{\eta}h\,,\quad \quad 
 \left( 
\begin{array}{c}
\dot s\\
\,\\
\dot \sigma
\end{array}
\right)
=
\mathbb{J}
\left( \begin{array}{c}
\frac{\partial h }
{\partial s}  \\
\,\\
\frac{\partial h }
{\partial \sigma} 
\end{array}
\right) \label{can_system}
\end{align}
whereas the reconstruction equation is
\begin{align}
\dot R &= R 
\left( 
\begin{array}{ccc}
0\\
0\\
\frac{\partial h}{\partial \nu}\,.
\end{array}
\right) \label{can_recons}
\end{align}

The reconstruction equation can be integrated to give rotations about the $z$ axis by angle:
  \begin{equation}
  \label{recon_eta}
\theta(t_0) = 
\int_0^{t_0} \frac{\partial h}{\partial \nu}\Big|_{\left(\nu_0, \eta(t), s(t), \sigma(t) \right)}\,dt.
\end{equation}

\medskip

We consider relative equilibria at
$(\nu,\eta,s,\sigma)= (\nu_0, 0,0,0)$, with velocity 
\begin{align} \label{E:releqvel}
\xi_0 := 
\left( 
\begin{array}{ccc}
0\\
0\\
\frac{\partial h}{\partial \nu}\Big|_{(\nu_0, 0,0,0)}
\end{array}
\right).
\end{align}

By construction 
  the Marsden-Weinstein reduced space at $\mu_0$ is locally symplectomorphic to the   canonical vector space  $so(3)_{\mu_0}^\perp \times T^*S$, and the dynamics are given by the reduced Hamiltonian $h_{\mu_0}(\eta, s, \sigma):= h(0, \eta, s, \sigma).$ In this model of the reduced space, the relative equilibrium $q_0$ becomes the origin.

\medskip
Recall that a simple mechanical system is a system with a Hamiltonian  $H: T^*Q \to \mathbb{R}$ of the form 
\begin{align}\label{E:simplemech}
H(q, p_q) =  \frac{1}{2}  {\mathbb{K}}^{-1} (p_q, p_q) + V(q)
\end{align}
for some $G$-invariant Riemannian metric $\mathbb{K}$, and some $G$-invariant 
potential $V: Q \to \mathbb{R}.$ We assume that $G$ acts properly. The dynamics on $T^*Q$  may be specialised easily this case. We take  $Q$ a finite dimensional vector space which, 
without loss of generality, we consider to be an open subset of $ \mathbb{R}^n.$

Fix $q_0 \in Q$ and let $N$ be the orthogonal complement to the group orbit through $q_0$.
By Palais' slice theorem, there is a neighbourhood $S$ of $0\in N$ such that the 
map $\tau: SO(3) \times S \to Q$, $(R, s) \mapsto R(q_0 + s)$ is a diffeomorphism onto its image.
The cotangent lift of $\tau$ defines local coordinates 
$(R,\xi, s, \dot s)$ on $TQ$,
$(R,\mu, s, \sigma)$ on $T^*Q$,
as explained in Section \ref{ssect:arbitrary}.
Since the Riemannian metric $\mathbb{K}$ and the potential $V$ are both $SO(3)$-invariant, 
they are independent of $R$ in these coordinates.
We write $\mathbb{K}(s)$ in block form in coordinates $(\xi, \dot s)$ as follows
(this defines $\mathbb{I}, \mathbb{C}$ and $m$):
\[
\mathbb{K}(s) = \left[\begin{array}{cc}
\mathbb{I}(s) &\,\,\, \mathbb{C}(s) \\
\,\,\,\mathbb{C}^T (s) &\,\,\, m(s)
\end{array}
\right]
\]
and define
$\mathbb{A}:=\mathbb{I}^{-1}\mathbb{C}$ and $\mathbb{M}:= m -
\mathbb{C}^T\mathbb{I}^{-1} \mathbb{C}$.
With these definitions, it can be shown  that 
the Hamiltonian takes the following form, 
 \begin{equation}h(\mu, s, \sigma)={\frac{1}{2}}\,
\mu^T {\mathbb{I}^{-1}}\, \mu + {\frac{1}{2}}
(\sigma-{\mathbb{A}}^{T}\mu )^T
{\mathbb{M}}^{-1}(\sigma-{\mathbb{A}}^{T}\mu) +V(s).
\label{H-bar-sms}\end{equation}
This is a special case, for free actions, of a more general result in \cite[Section 6]{RSS06}. In the case of $N$-body problems (molecules), the corresponding  Hamiltonian is deduced in \cite{MoRo99} and \cite{Ciftci12}. The relative equilibria conditions are \cite{RSS06}
\begin{align*}
&\sigma=  {\mathbb{A}}^{T}(s)\mu\\
&\mu \times \left( {\mathbb{I}^{-1}}(s)\, \mu \right)=0\\
&\frac{\partial}{\partial s}\left( V(s) + \mu^T {\mathbb{I}^{-1}}(s)\, \mu\right)=0\,.
\end{align*}

Assume that $\left((0,0,\mu_0), s_0, \sigma_0 \right)$ is a relative equilibrium, as determined by the
above equations, and let $(q_0, p_0) \simeq \left(Id, (0,0,\mu_0), s_0, \sigma_0 \right)$.  Now we  
express the Hamiltonian \eqref{H-bar-sms} in the slice coordinates given by \eqref{mu_1} -- \eqref{mu_3}, obtaining
\begin{align*}
h: so(3)_{\mu_0} \times so(3)_{\mu_0}^\perp \times T^*S &\to \mathbb{R}\,,\quad 
h= h(\nu,\eta, s, \sigma)\,. 
\end{align*}
The equations of motion are given by \eqref{can_system_1} -- \eqref{can_recons}.
%
%
%
%

\subsection{A conjecture on the Nekhoroshev's estimates near a relative equilibrium}
\label{ssect:Nehkoroshev}

Recall  that in \cite{Nek77} Nekhoroshev showed that under a perturbation of order $\varepsilon$, the actions of an arbitrary orbit of a quasi-convex  integrable Hamiltonian vary at order $\varepsilon^b$ over  a time interval of order $\text{exp}(\varepsilon^{-a}),$ where   $a$ and $b$ are positive numbers which depend on the number of degrees of freedom and the steepness of the Hamiltonian. For Hamiltonians near an elliptic equilibrium, under certain hypothesis, analogous estimates are found  by \cite{Fasso98}, \cite{Poschel99} and \cite{Niederman98}. Specifically, (under the right conditions) the actions $ I = (I_1, I_2, \ldots, I_n)$ of a  Hamiltonian system near an elliptic equilibrium fulfill 
\[ \left|  I(t) -  I(0) \right| < C\varepsilon^a\quad \quad \text{for}\,\,\, |t|< D_1 \text{exp}(D_2 \varepsilon^{-a})  \]
where $a, C,  D_1, D_2$ are constants independent of $\varepsilon$. 

\medskip
 The  Nekhoroshev long term stability of the perturbed Euler-Poinsot rigid-body  near a relative equilibrium (proper rotation)  was treated in  a series of excellent papers by Benettin, Fass\`o et al. (see \cite{Fasso_4} and references therein).  In these papers, the authors show that for a perturbed rigid body,  the proper rotations around the symmetry axis are Nekhoroshev stable. 

Recall from Section \ref{ssect:freerigid} that our slice coordinates for the reduced space correspond to the regularised 
Serret-Andoyer-Deprit coordinates used by 
Benettin, Fass\`o et al.
This suggests that our slice coordinates for general 
rotationally-invariant cotangent-bundle systems may be useful for addressing the conjecture
that in all such systems,
the (non-linearly) stable relative equilibria  are Nekhoroshev long term stable.

\section{Birkhoff-Poincar\'e normal forms near a relative equilibrium}\label{sect:normalforms}

\subsection{Rotationally invariant cotangent bundle systems}

Consider a canonical symplectic manifold $(P, \Omega_c)$, a Hamiltonian $H: P \to \mathbb{R}$ and $z_0$ an equilibrium of the dynamics induced by $H$. Denote by 
 $\hat H^{(i)}$ the homogeneous polynomial of degree $i$ as obtained from the Taylor expansion of $H$ around $z_0.$
  The truncated normal form of order $k$ is defined as   the $k$-jet of the  Hamiltonian written in some  (new) coordinates  $\hat H$
 \[
 j^k\hat H = \hat H^{(2)} +  \hat H^{(3)} +\ldots + \hat H^{(k)}
 \]
which fulfills 
\[
\left\{  
\hat H^{(2)} \,, \hat H^{(i)}\right\} =0 \quad \quad \text{for all} \,\,\, i=2,3,\ldots k
\]
The method itself consists in obtaining the property above  by applying iteratively changes of coordinates as given by  the time-1 Hamiltonian flow $X^1_F$ where $F$ is a homogeneous polynomial of degree $k$ found by solving the \textit{homological} equation
\[
\hat H^{(k)} + \left\{  
\hat H^{(2)} \,, F 
\right\}=0
\]
A detailed exposition of Hamiltonian normal forms can be found in \cite{Broer09} (see also \cite{MoRa05}). 

\medskip
The classical  method of Birkhoff-Poincar\'e normal forms near an equilibrium can now be applied to the study of dynamics  near  relative equilibria in the reduced space in the case of rotationally invariant systems.  Indeed, since a relative equilibrium is an equilibrium in the reduced space, and since the reduced space is endowed with a canonical form, one may immediately apply the standard theory.

\medskip
We will not  report here on the importance and usefulness  of  normal forms in relation, for instance,  to bifurcation and stability theory (the interested reader may consult, for instance, \cite{Broer09} and references therein, as well as \cite{M92}). In the context of cotangent bundles systems with $SO(3)$ symmetry, very recent applications can be found in  \cite{Ciftci12} and \cite{Ciftci14}; here, since the normal forms are calculated  directly on the reduced space, there is no need of a canonical  embedding in the full space.  

\subsection{The general case}

All of the theory in  Section \ref{Normal_forms_SO(3)} can be generalised to arbitrary Lie groups, i.e., to
proper, cotangent-lifted free actions of any $G$ on $T^*G$.
A key difference is that we have no general formula for the symplectic tube $\Phi$, and do not expect to find one. Thus we do not expect to be able to write $H$ explicitly in slice coordinates
(though this might be possible in special cases).
Nonetheless, the equations of motion in the slice have almost the same form as in 
\eqref{nu_const} and \eqref{zeta_eq}. If $G_\mu$ is compact, the equations of motion are
\begin{align}\label{general_sliceeq}
\dot \nu &= \ad\nolimits^*_{\frac{\partial h}{\partial \nu}} \nu,  \\
\dot \eta &= -\frac{1}{\mu_0} \mathbb{J} \,\nabla_\eta h,
\label{general_zetaeq}
\end{align}
with the first equation reducing to $\dot \nu = 0$ whenever $G_\mu$ is abelian.
The reconstruction equation takes the same form as before:
\begin{align}\label{general_recons}
R(t)^{-1} \dot R(t) = \frac{\partial h}{\partial \nu}.
\end{align}
The case of non-compact $G_\mu$ is dealt with in \cite{RWL02}.
For simple mechanical systems, the relative equilibrium conditions given in Section \ref{ssect:CBrot}
have a generalisation in \cite{RSS06}.


It is clear that if an explicit  formula for the symplectic tube $\Phi: G \times \mathfrak{g}_{\mu} \times N \to G \times \mathfrak{g}^*$ exists, then this can be composed with the original Hamiltonian to express it in slice coordinates, and this Hamiltonian can then be differentiated as needed.
However, 
a key observation is that,
to obtain a truncated normal form of order $k$, it is not necessary to have an explicit formula for $\hat H$;
all that is required is its truncated series expansion. 
In particular, to obtain such a truncation in slice coordinates one needs only the derivatives of $\Phi$ at $(e,0,0)$ up to order $k$. 
These can be obtained using the Tube Condition in Proposition \ref{tubecond}, which for reader's convenience we re-write:


\medskip
\noindent
\textit{
$\Phi^* \Omega_c = \Omega_Y$ if and only if
\begin{align*}
\Phi(g, \nu, \ad\nolimits_\eta^* \mu) 
= \left(
g F(\nu, \eta)^{-1}, 
\Ad\nolimits_{F(\nu, \eta)^{-1}}^*
\left(\mu + \nu\right)
\right)
\end{align*}
for some $F:\mathfrak{g}_\mu^* \times \mathfrak{g}_\mu^\perp \to G$ such that $F(0,0) = e$ and,
for all $\nu, \eta, \nu_i,\zeta_i$,
\begin{align*}
&\left\langle 
\mu + \nu, 
\left[ 
F(\nu, \eta)^{-1} \left(DF(\nu, \eta) \cdot \left(\dot \nu_1, \zeta_1\right)\right),
F(\nu, \eta)^{-1} \left(DF(\nu, \eta) \cdot \left(\dot \nu_2, \zeta_2\right)\right)
\right]
\right\rangle \\
&+
\left\langle
\dot \nu_2, F(\nu, \eta)^{-1} \left(DF(\nu, \eta) \cdot \left(\dot \nu_1, \zeta_1\right)\right)
\right\rangle
- \left\langle
\dot \nu_1, F(\nu, \eta)^{-1} \left(DF(\nu, \eta) \cdot \left(\dot \nu_2, \zeta_2\right)\right)
\right\rangle \notag \\
&= \left\langle \mu, \left[\zeta_1, \zeta_2\right]\right\rangle. \notag
\end{align*}
}
For any specific matrix Lie group $G$,
this condition can be solved directly for $DF(0,0)$, while implicit differentiation of the 
same condition allows the \textit{iterative} calculation of higher derivatives to the desired order.  
Note that there may not be unique solutions to these equations, since the symplectic tube
is in general not unique. Different choices of solutions will lead to different normal forms, all valid.

Note that the Lie symmetry group need not be compact, either in the Tube Condition in Proposition \ref{tubecond} or in Lemma \ref{lem:ansatz}. In particular, one can apply the methodology outlined here for $G=SE(3)$ for which an interesting  case study is given by  the so-called full two body body problem, that is, two spatially extended bodies, (two asteroids), in gravitational interaction. We intend to investigate such problems in the future.

\section{Relationship to Reduced Energy Momentum method}
\label{sect:REM}

We consider the relationship between the symplectic slice coordinates studied here and the Reduced
Energy Momentum (REM) method \cite{SLM91, M92}.
The general Energy-Momentum Method  \cite{Pat92} concerns a relative equilibrium $z_e$,
with velocity $\xi_e$, of 
a $G$-symmetric Hamiltonian system. The method gives sufficient conditions for
for a kind of equivariant nonlinear stability called $G_{\mu_e}$-stability, where $\mu_e = J(z_e)$.
The main condition is that the \textit{augmented Hamiltonian} defined by
$H_{\xi_e}(z) = H(z)  - \left<J(z), \xi_e \right>$ be definite on some (and hence any) 
subspace $\mathcal{S}$ of $\ker dJ(z_e)$ 
that is 
transverse to $\mathfrak{g}_{\mu_e} z_e$.

Consider a simple mechanical system on $T^*Q$, with Hamiltonian as in (\ref{E:simplemech})
and $G$ acting properly, with relative equilibrium $z_e = (q_e, p_e)$.
The REM reduces the main condition of the Energy-Momentum method to two simple tests of definiteness 
on subspaces of $T_{q_e}Q$. 
This provides a computationally cheap way to prove nonlinear stability in some cases.
The proof of the REM relies on a particular choice of the subspace $\mathcal{S}$ mentioned above,
and a particular splitting of that subspace that block-diagonalises $d^2H_{\xi_e}$ .
We compute some of these spaces in coordinates given the Palais slice theorem for the action of $G$ on $Q$.
Let $S$ be a slice in $Q$ at $q_e$ with respect to the given metric.
Without loss of generality we assume S is a vector space,
so that
\begin{align*}
T^*Q \cong T^*(G\times S) &\cong \mathfrak{g} \oplus \mathfrak{g}^* \oplus S \oplus S^*,
\end{align*}
where we use left-trivialisation on $T^*G$.
The relative equilibrium in the new coordinates is $z_e = (e,\mu,0,0)$.
Restricting the Riemannian metric at $q_e = (e,0)$ to the subspace $\mathfrak{g} \oplus \{0\}$ of 
$T_e(G \times S) \cong \mathfrak{g} \oplus S$
gives an inner product on $\mathfrak{g}$, 
with respect to which we take the complement $\mathfrak{g}^\perp$.
\footnote{This inner product need not be invariant with respect to the
adjoint action of $G_\mu$ on $\mathfrak{g}$.
One of the conditions of the Energy-Momentum Method is that 
$\mathfrak{g}$ admits a $G_{\mu_e}$-invariant inner product.} 
We calculate the spaces $\mathcal{V}\subset T_{(e,0)} Q$ and 
$\mathcal{S}\subset \ker dJ(z_e)$ in the Reduced Energy-Momentum method
as presented in \cite{M92}:
\begin{align*}
J(g,\nu, s, \sigma) &= \Ad\nolimits_{g^{-1}}^* \nu, \\
\ker dJ(e,\mu, 0, 0) &= \{(\eta, \ad\nolimits_\eta^* \mu) : \eta \in \mathfrak{g}\} 
\oplus S \oplus S^*,\\
T_{z_e}(G_\mu z_e) = \mathfrak{g}_\mu z_e &= \mathfrak{g}_\mu \oplus \{0\} \oplus \{0\} \oplus \{0\}, \\
\mathcal{V} &:= \left(\mathfrak{g}_\mu (e,0)\right)^\perp
= \left(\mathfrak{g}_\mu \oplus \{0\}\right)^\perp = \mathfrak{g}_\mu^\perp \oplus S, 
\end{align*}
where the last equality is due to the definition of the inner product on $\mathfrak{g}$, and
\begin{align*}
\mathcal{S} &:=
\{\delta z \in \ker DJ(z_e) : T\pi_Q \cdot \delta z \in \mathcal{V}\} \\
&= \{(\eta, \ad\nolimits_\eta^* \mu) : \eta \in \mathfrak{g}_\mu^\perp\} \oplus S \oplus S^* \\
& = N_1 \oplus S \oplus S^*,
\end{align*}
where $N_1$ is as in \eqref{E:N1}. 
Since $\mathcal{S}$ is a complement to $\mathfrak{g}_\mu z_e$ in $\ker DJ(z_e)$, it is
a realisation of the symplectic normal space $N_s$, and in fact it is the 
same as the realisation that appears in the constructive symplectic slice theorem in Section \ref{ssect:arbitrary}
(recall that $N_1 \cong T_\mu\mathcal{O}_\mu \cong\mathfrak{g}_\mu^\perp$). 
The REM splits $\mathcal{S}$ further:
\begin{align*}
\mathcal{S} = \mathcal{S}_{RIG} \oplus \mathcal{W}_{INT} \oplus \mathcal{W}^*_{INT}.
\end{align*}
We will not fully calculate these spaces here, but the following can easily be checked:
\begin{align}\label{E:REMsplit}
\mathcal{S}_{RIG} &=  N_1 \oplus \{0\} \oplus \{0\},\\
\mathcal{W}_{INT} &\le \mathfrak{g}_\mu^\perp \oplus \mathfrak{g}^* \oplus S \oplus \{0\}, \notag\\
\mathcal{W}_{INT}^* &= \{0\} \oplus \{0\} \oplus \{0\} \oplus S^*. \notag
\end{align}
The REM works in part because this splitting of $\mathcal{S}$ block-diagonalises 
the augmented Hamiltonian.

In contrast, the Hamiltonian Slice Theorem 
block-diagonalises the \textit{symplectic form},
and it does so at every point $z$, not just $z_e$.
In symplectic slice coordinates,
the symplectic form
block-diagonalises 
with respect to the two-way splitting $\left(\mathfrak{g} \oplus \mathfrak{g}_\mu^*\right) \oplus N_s$,
where 
\begin{align}\label{E:CBSsplit}
N_s \cong  N_1 \oplus S \oplus S^* \cong T_\mu\mathcal{O}_\mu  \oplus S \oplus S^*
\end{align}
and $\Omega_{N_s}$ has the following form with respect to this
splitting: 
\begin{align}\label{E:CBSomegaNs}
\left[\begin{array}{ccc}
\Omega_{KKS} & 0 & 0 \\
0 & 0 & I \\
0 & -I & 0
\end{array}\right].
\end{align}
Thus the total symplectic form block-diagonalises with respect to the 
the 3-way splitting $\left(\mathfrak{g} \oplus \mathfrak{g}_\mu^*\right) \oplus N_1 \oplus \left(S \oplus S^*\right)$.

The REM and the constructive Hamiltonian Slice Theorem
both make use of the same realisation of the symplectic normal space, 
$\mathcal{S} = N_1 \oplus S \oplus S^*$, but
while the slice theorem uses the canonical 3-way splitting
$N_1 \oplus S \oplus S^*$,
the REM uses the splitting in \eqref{E:REMsplit}.
The two splittings do share one common subspace, $\{0\}\oplus \{0\} \oplus \{0\} \oplus S^*$, however there the similarities end. 
The splitting in the REM is chosen to block-diagonalise the augmented Hamiltonian,
leading to a stability condition defined directly on configuration space.
The splitting in the constructive Hamiltonian Slice Theorem
puts the symplectic form into block form, but not the augmented Hamiltonian,
and is not associated with a convenient condition for nonlinear stability.

For the specific purpose of proving stability of a relative equilibrium of a simple mechanical system,
the REM is a superb tool. 
Symplectic slice coordinates are general-purpose symmetry-adapted coordinates
on phase space
that block-diagonalise the symplectic form at every $z$,
leading to a normal form for the Hamiltonian equations given in
\eqref{general_sliceeq}, \eqref{general_zetaeq} and \eqref{general_recons}.
The simple form of these equations, and the fact that \eqref{general_zetaeq} is
the reduced Hamiltonian system, make these coordinates ideal for 
computing Birkhoff-Poincar\'e normal forms.

\section{Acknowledgements}

\medskip
CS was supported by an NSERC Discovery grant. This work was completed during a  research stay at the Otter Lake Science Institute in Ontario. Also, we  thank the referee for many useful comments. 

\section{Appendix}\label{appendix}

This appendix contains proofs of three results in the main text.
The main result is Proposition \ref{tubecond} (the ``Tube Condition'')
in Section \ref{ssect:cotbunslice}, which gives a necessary and sufficient condition for a
map $\Phi$ from $G \times \mathfrak{g}_\mu^* \times N_s$ to $G\times \lieg^*$
to be symplectic.
This proposition is used in Section \ref{ssect:sliceSO3} to construct an explicit symplectic tube
when $G=SO(3)$, and it is also a key ingredient in the algorithm outlined in Section \ref {sect:normalforms}
for computing Birkhoff-Poincar\'e normal forms for arbitrary $G$.

In the statement of the proposition, $\Omega_c$ is the canonical symplectic form on $T^*G$, which
is identified by left-trivialisation with $G\times \lieg^*$. The symplectic form $\Omega_Y$ is
the form on $Y := G \times \mathfrak{g}_\mu^* \times N_s$ that appears in the Hamiltonian Slice Theorem
(Theorem \ref{marle}). This symplectic form is stated more explicitly in \eqref{E:OmegaY},
using the identification of $N_s$ with $T_\mu \mathcal{O}_\mu$ that appears earlier in the same section.

\medskip

\textbf{Restatement of Proposition \ref{tubecond} (Tube Condition)}:
$\Phi^* \Omega_c = \Omega_Y$ if and only if

\begin{align*}
\Phi(g, \nu, \ad\nolimits_\eta^* \mu) 
= \left(
g F(\nu, \eta)^{-1}, 
\Ad\nolimits_{F(\nu, \eta)^{-1}}^*
\left(\mu + \nu\right)
\right)
\end{align*}
for some $F:\mathfrak{g}_\mu^* \times \mathfrak{g}_\mu^\perp \to G$ such that 
%

\begin{align*}
&
\left\langle 
\mu + \nu, 
\left[ 
F(\nu, \eta)^{-1} \left(DF(\nu, \eta) \cdot \left(\dot \nu_1, \zeta_1\right)\right),
F(\nu, \eta)^{-1} \left(DF(\nu, \eta) \cdot \left(\dot \nu_2, \zeta_2\right)\right)
\right]
\right\rangle \\
&+
\left\langle
\dot \nu_2, F(\nu, \eta)^{-1} \left(DF(\nu, \eta) \cdot \left(\dot \nu_1, \zeta_1 \right) \right)
\right\rangle
- \left\langle
\dot \nu_1, F(\nu, \eta)^{-1} \left(DF(\nu, \eta) \cdot \left( \dot \nu_2, \zeta_2 \right) \right)
\right\rangle \notag \\
&= \left\langle \mu, \left[\zeta_1, \zeta_2\right] \right \rangle. \notag
\end{align*}

\begin{proof}
The most general formula for a $G$-equivariant $\Phi$ is
\begin{align*}
\Phi(g, \nu, \ad\nolimits_\eta^* \mu) 
= \left(g F_1(\nu, \eta), F_2(\nu, \eta)\right).
\end{align*}
We consider the condition $\Phi^*\Omega_c = \Omega_Y$.
Since
\begin{align*}
D\Phi(g, \nu, \ad\nolimits_\eta^* \mu)
\cdot (\xi,0,0,) 
= \left(\Ad\nolimits_{F_1^{-1}(\nu, \eta)}^*\xi , 0\right),
\end{align*}
it follows that, for all $\xi_1, \xi_2$,
\begin{align*}
\left(\Phi^*\Omega_c\right)(e, \nu, \ad\nolimits_\eta^* \mu)
\left(
\left(\xi_1, 0,0 \right)
\left(\xi_2, 0,0\right)
\right)
&=
\Omega_Y(e, \nu, \ad\nolimits_\eta^* \mu)
\left(
\left(\xi_1, 0,0\right)
\left(\xi_2, 0,0\right)
\right) \\
\Leftrightarrow 
\left\langle F_2\left(\nu, \ad\nolimits_\eta^* \mu\right),
\left[\Ad\nolimits_{F_1^{-1}(\nu, \eta)}\xi_1, \Ad\nolimits_{F_1^{-1}(\nu, \eta)}\xi_2 \right] 
\right\rangle
&=
\left\langle \mu + \nu, [\xi_1,\xi_2]\right\rangle \\
\Leftrightarrow 
\left\langle  \Ad\nolimits_{F_1^{-1}(\nu, \eta)}^* F_2\left(\nu, \ad\nolimits_\eta^* \mu\right),
\left[\xi_1, \xi_2 \right] 
\right\rangle
&=
\left\langle \mu + \nu, [\xi_1,\xi_2]\right\rangle.
\end{align*}
Hence this condition is true for all $\xi_1, \xi_2$ if and only if
\begin{align*}
F_2\left(\nu, \ad\nolimits_\eta^* \mu\right)
&= \Ad\nolimits_{F_1(\nu, \eta)}^*\left(\mu + \nu\right).
\end{align*}
Let $F = F_1^{-1}$. We will use the notation 
$F(\nu, \eta) \nu :=  \Ad\nolimits_{F(\nu, \eta)^{-1}}^* \nu$, so 
\begin{align*}
\Phi(g, \nu, \ad\nolimits_\eta^* \mu) 
= \left(
g F(\nu, \eta)^{-1}, F(\nu, \eta) \left(\mu + \nu\right)
\right),
\end{align*}
with first derivative:
\begin{align*}
D\Phi&(g, \nu, \ad\nolimits_\eta^* \mu)
\cdot (\xi, \dot \nu, \ad\nolimits_\zeta^* \mu) \\
&= \left(\Ad\nolimits_{F(\nu, \eta)}\xi 
-  \left( DF(\nu, \eta) \cdot (\dot \nu, \zeta)\right) F(\nu, \eta)^{-1}, 
\left(DF(\nu, \eta) \cdot (\dot \nu, \zeta)\right)(\mu + \nu) + F(\nu, \eta) \,  \dot \nu\right).
\end{align*}
(using left-trivialisation in the first component).
Hence,

\begin{align*}
&\left(\Phi^*\Omega_c\right)(e, \nu, \ad\nolimits_\eta^* \mu)
\left(
\left(\xi_1, \dot \nu_1, \ad\nolimits_{\zeta_1}^* \mu\right)
\left(\xi_2, \dot \nu_2, \ad\nolimits_{\zeta_2}^* \mu\right)
\right) \\
&= \Omega_c
\left(
g F(\nu, \eta)^{-1}, 
F(\nu, \eta)
\left(\mu + \nu\right)
\right) \\
& \quad  (
\left(\Ad\nolimits_{F(\nu, \eta)}\xi_1 
-  \left( DF(\nu, \eta) \cdot (\dot \nu_1, \zeta_1)\right) F(\nu, \eta)^{-1}, 
\left(DF(\nu, \eta) \cdot (\dot \nu_1, \zeta_1)\right)(\mu + \nu) + F(\nu, \eta) \,  \dot \nu_1\right), \\
& \qquad 
\left(\Ad\nolimits_{F(\nu, \eta)}\xi_2 
-  \left( DF(\nu, \eta) \cdot (\dot \nu_2, \zeta_2)\right) F(\nu, \eta)^{-1}, 
\left(DF(\nu, \eta) \cdot (\dot \nu_2, \zeta_2)\right)(\mu + \nu) + F(\nu, \eta) \,  \dot \nu_2\right)
)
\\
&= 
\langle F\left(\nu, \eta\right)(\mu + \nu), \\
&\left[ 
\Ad\nolimits_{F(\nu, \eta)}\xi_1 
-  \left( DF(\nu, \eta) \cdot (\dot \nu_1, \zeta_1)\right) F(\nu, \eta)^{-1},
\Ad\nolimits_{F(\nu, \eta)}\xi_2 
-  \left( DF(\nu, \eta) \cdot (\dot \nu_2, \zeta_2)\right) F(\nu, \eta)^{-1}
\right] 
\rangle \\
&\quad 
+ \left\langle 
\left(DF(\nu, \eta) \cdot (\dot \nu_2, \zeta_2)\right)(\mu + \nu) + F(\nu, \eta) \,  \dot \nu_2, 
\Ad\nolimits_{F(\nu, \eta)}\xi_1 
-  \left( DF(\nu, \eta) \cdot (\dot \nu_1, \zeta_1)\right) F(\nu, \eta)^{-1}\right\rangle \\
&\quad - \left\langle 
\left(DF(\nu, \eta) \cdot (\dot \nu_1, \zeta_1)\right)(\mu + \nu) + F(\nu, \eta) \,  \dot \nu_1, 
\Ad\nolimits_{F(\nu, \eta)}\xi_2 
-  \left( DF(\nu, \eta) \cdot (\dot \nu_2, \zeta_2)\right) F(\nu, \eta)^{-1}
\right\rangle.
\end{align*}
To verify the condition $\Phi^*\Omega_c = \Omega_Y$, we must consider all pairs of 
 tangent vectors $(\xi_i, \dot \nu_i, \zeta_i)$. 
By linearity, it suffices to consider only tangent vectors where two of these
three components are zero. Thus there are 9 types of tangent vector pairs to consider,
which reduce to 6 types by skew-symmetry.
The \underline{$\xi - \xi$} case has already been considered above,
with the conclusion that the pull-back condition is automatically satisfied for arbitrary $F$.
This same conclusion will now be shown to apply in the  
the \underline{$\xi - \dot\nu$}  and \underline{$\xi - \zeta$} cases.
Finally, we will combine the remaining 3 cases into one \underline{$(\nu, \zeta) - (\nu, \zeta)$} case,
which will lead to the Tube Condition in Proposition \ref{tubecond}.

\medskip
\noindent \underline{Case $\xi-\dot \nu$}: 
\begin{align*}
&\left(\Phi^*\Omega_c\right)(e, \nu, \ad\nolimits_\eta^* \mu)
\left(
\left(\xi_1, 0,0\right)
\left(0, \dot \nu_2,0\right)
\right) \\
&=\left\langle F\left(\nu, \eta\right)(\mu + \nu),
\left[ 
\Ad\nolimits_{F(\nu, \eta)}\xi_1,
-  \left( DF(\nu, \eta) \cdot (\dot \nu_2, 0)\right) F(\nu, \eta)^{-1}
\right] 
\right\rangle \\
&\quad 
+ \left\langle 
\left(DF(\nu, \eta) \cdot (\dot \nu_2, 0)\right)(\mu + \nu) + F(\nu, \eta) \,  \dot \nu_2, 
\Ad\nolimits_{F(\nu, \eta)}\xi_1 
\right\rangle\\
&=\left\langle F\left(\nu, \eta\right)(\mu + \nu),
\left[ 
\left( DF(\nu, \eta) \cdot (\dot \nu_2, 0)\right) F(\nu, \eta)^{-1},
\Ad\nolimits_{F(\nu, \eta)}\xi_1
\right] 
\right\rangle \\
&\quad 
+ \left\langle 
\left(DF(\nu, \eta) \cdot (\dot \nu_2, 0)\right)(\mu + \nu), 
\Ad\nolimits_{F(\nu, \eta)}\xi_1 
\right\rangle\\
&\quad
+ \left\langle 
F(\nu, \eta) \,  \dot \nu_2, 
\Ad\nolimits_{F(\nu, \eta)}\xi_1 
\right\rangle
\end{align*}
Using explicit notation for the coadjoint action gives:
\begin{align*}
&\left\langle \Ad\nolimits_{F(\nu, \eta)^{-1}}^* (\mu + \nu),
\ad\nolimits_{\left( DF(\nu, \eta) \cdot (\dot \nu_2, 0)\right) F(\nu, \eta)^{-1}}
\left(\Ad\nolimits_{F(\nu, \eta)}\xi_1\right)
\right\rangle \\
&\quad 
+ \left\langle 
- \Ad\nolimits_{F(\nu, \eta)^{-1}}^*
\ad\nolimits^*_{F(\nu, \eta)^{-1}\left(DF(\nu, \eta) \cdot (\dot \nu_2, 0)\right)}(\mu + \nu), 
\Ad\nolimits_{F(\nu, \eta)}\xi_1 
\right\rangle\\
&\quad + \left\langle 
\Ad\nolimits_{F(\nu, \eta)^{-1}}^* \,  \dot \nu_2, 
\Ad\nolimits_{F(\nu, \eta)}\xi_1 
\right\rangle \\
&=\left\langle \Ad\nolimits_{F(\nu, \eta)^{-1}}^* (\mu + \nu),
\ad\nolimits_{\left( DF(\nu, \eta) \cdot (\dot \nu_2, 0)\right) F(\nu, \eta)^{-1}}
\left(\Ad\nolimits_{F(\nu, \eta)}\xi_1\right)
\right\rangle \\
&\quad 
- \left\langle 
\ad\nolimits^*_{\left(DF(\nu, \eta) \cdot (\dot \nu_2, 0)\right)F(\nu, \eta)^{-1}}
\left(\Ad\nolimits_{F(\nu, \eta)^{-1}}^*(\mu + \nu) \right), 
\Ad\nolimits_{F(\nu, \eta)}\xi_1 
\right\rangle\\
&\quad + \left\langle 
\dot \nu_2, 
\xi_1 
\right\rangle \\
&=\left\langle 
\dot \nu_2, 
\xi_1 
\right\rangle \\
&=\Omega_Y(e, \nu, \ad\nolimits_\eta^* \mu)
\left(
\left(\xi_1, 0,0\right)
\left(0, \dot \nu_2,0\right)
\right), 
\qquad \textrm{for all } \xi_1,\dot\nu_2 
\end{align*}
automatically, for all functions $F$.

\medskip
\noindent \underline{Case $\xi-\zeta$}: 
\begin{align*}
&\left(\Phi^*\Omega_c\right)(e, \nu, \ad\nolimits_\eta^* \mu)
\left(
\left(\xi_1, 0,0\right),
\left(0, 0,\ad\nolimits_{\zeta_2}^* \mu\right)
\right) \\
&=\left\langle F\left(\nu, \eta\right)(\mu + \nu),
\left[ 
\Ad\nolimits_{F(\nu, \eta)}\xi_1,
-  \left( DF(\nu, \eta) \cdot (0, \zeta_2)\right) F(\nu, \eta)^{-1}
\right] 
\right\rangle \\
&\quad 
+ \left\langle 
\left(DF(\nu, \eta) \cdot (0, \zeta_2)\right)(\mu + \nu), 
\Ad\nolimits_{F(\nu, \eta)}\xi_1 
\right\rangle\\
&=0 \\
&=\Omega_Y(e, \nu, \ad\nolimits_\eta^* \mu)
\left(
\left(\xi_1, 0,0\right)
\left(0, 0, \zeta_2\right)
\right), 
\qquad \textrm{for all } \xi_1,\zeta_2 
\end{align*}
automatically, for all functions $F$.

\medskip
\noindent
\underline{Case $(\dot \nu,\zeta) - (\dot \nu,\zeta)$} (three cases combined)
\begin{align*}
&\left(\Phi^*\Omega_c\right)(e, \nu, \ad\nolimits_\eta^* \mu)
\left(
\left(0, \dot \nu_1, \ad\nolimits_{\zeta_1}^* \mu\right)
\left(0, \dot \nu_2, \ad\nolimits_{\zeta_2}^* \mu\right)
\right) \\
&= 
\left\langle F\left(\nu, \eta\right)(\mu + \nu),
\left[ 
\left( DF(\nu, \eta) \cdot (\dot \nu_1, \zeta_1)\right) F(\nu, \eta)^{-1},
 \left( DF(\nu, \eta) \cdot (\dot \nu_2, \zeta_2)\right) F(\nu, \eta)^{-1}
\right] 
\right\rangle \\
&\quad 
-\left\langle 
\left(DF(\nu, \eta) \cdot (\dot \nu_2, \zeta_2)\right)(\mu + \nu), 
\left( DF(\nu, \eta) \cdot (\dot \nu_1, \zeta_1)\right) F(\nu, \eta)^{-1}\right\rangle \\
&\quad 
+\left\langle 
\left(DF(\nu, \eta) \cdot (\dot \nu_1, \zeta_1)\right)(\mu + \nu), 
\left( DF(\nu, \eta) \cdot (\dot \nu_2, \zeta_2)\right) F(\nu, \eta)^{-1}
\right\rangle \\
&\quad 
- \left\langle 
F(\nu, \eta) \,  \dot \nu_2, 
\left( DF(\nu, \eta) \cdot (\dot \nu_1, \zeta_1)\right) F(\nu, \eta)^{-1}
\right\rangle \\
&\quad 
+ \left\langle 
F(\nu, \eta) \,  \dot \nu_1, 
\left( DF(\nu, \eta) \cdot (\dot \nu_2, \zeta_2)\right) F(\nu, \eta)^{-1}
\right\rangle.
\end{align*}
Using explicit notation for the coadjoint action gives:
\begin{align*}
&\left(\Phi^*\Omega_c\right)(e, \nu, \ad\nolimits_\eta^* \mu)
\left(
\left(0, \dot \nu_1, \ad\nolimits_{\zeta_1}^* \mu\right)
\left(0, \dot \nu_2, \ad\nolimits_{\zeta_2}^* \mu\right)
\right) \\
&= 
\left\langle \Ad\nolimits_{F(\nu, \eta)^{-1}}^*(\mu + \nu),
\left[
\left( DF(\nu, \eta) \cdot (\dot \nu_1, \zeta_1)\right) F(\nu, \eta)^{-1},
 \left( DF(\nu, \eta) \cdot (\dot \nu_2, \zeta_2)\right) F(\nu, \eta)^{-1}
 \right]
\right\rangle \\
&\quad 
-\left\langle 
\ad\nolimits^*_{\left(DF(\nu, \eta) \cdot (\dot \nu_2, \zeta_2)\right)F(\nu, \eta)^{-1}}
\left(\Ad\nolimits_{F(\nu, \eta)^{-1}}^*(\mu + \nu) \right), 
\left( DF(\nu, \eta) \cdot (\dot \nu_1, \zeta_1)\right) F(\nu, \eta)^{-1}\right\rangle \\
&\quad 
+\left\langle 
\ad\nolimits^*_{\left(DF(\nu, \eta) \cdot (\dot \nu_1, \zeta_1)\right)F(\nu, \eta)^{-1}}
\left(\Ad\nolimits_{F(\nu, \eta)^{-1}}^*(\mu + \nu) \right), 
\left( DF(\nu, \eta) \cdot (\dot \nu_2, \zeta_2)\right) F(\nu, \eta)^{-1}
\right\rangle \\
&\quad 
- \left\langle 
\dot \nu_2, 
F(\nu, \eta)^{-1} \left(  DF(\nu, \eta) \cdot (\dot \nu_1, \zeta_1)\right)
\right\rangle \\
&\quad 
+ \left\langle 
\dot \nu_1, 
 F(\nu, \eta)^{-1} \left( DF(\nu, \eta) \cdot (\dot \nu_2, \zeta_2)\right)
\right\rangle\\
&= 
-\left\langle \Ad\nolimits_{F(\nu, \eta)^{-1}}^*(\mu + \nu),
\left[
\left( DF(\nu, \eta) \cdot (\dot \nu_1, \zeta_1)\right) F(\nu, \eta)^{-1},
 \left( DF(\nu, \eta) \cdot (\dot \nu_2, \zeta_2)\right) F(\nu, \eta)^{-1}
 \right]
\right\rangle \\
&\quad 
- \left\langle 
\dot \nu_2, 
F(\nu, \eta)^{-1} \left(  DF(\nu, \eta) \cdot (\dot \nu_1, \zeta_1)\right)
\right\rangle 
+ \left\langle 
\dot \nu_1, 
 F(\nu, \eta)^{-1} \left( DF(\nu, \eta) \cdot (\dot \nu_2, \zeta_2)\right)
\right\rangle\\
&= 
-\left\langle \mu + \nu,
\left[
F(\nu, \eta)^{-1} \left( DF(\nu, \eta) \cdot (\dot \nu_1, \zeta_1)\right),
F(\nu, \eta)^{-1} \left( DF(\nu, \eta) \cdot (\dot \nu_2, \zeta_2)\right)
\right]
\right\rangle \\
&\quad 
- \left\langle 
\dot \nu_2, 
F(\nu, \eta)^{-1} \left(  DF(\nu, \eta) \cdot (\dot \nu_1, \zeta_1)\right)
\right\rangle 
+ \left\langle 
\dot \nu_1, 
 F(\nu, \eta)^{-1} \left( DF(\nu, \eta) \cdot (\dot \nu_2, \zeta_2)\right)
\right\rangle.
\end{align*}
We need this to equal
$\Omega_Y(e, \nu, \ad\nolimits_\eta^* \mu)
\left(
\left(0, \dot \nu_1, \ad\nolimits_{\zeta_1}^* \mu\right)
\left(0, \dot \nu_2, \ad\nolimits_{\zeta_2}^* \mu\right)
\right)
= - \left\langle \mu, \left[\zeta_1, \zeta_2\right]\right\rangle$,
for all $\nu, \dot \nu_1, \dot \nu_2 \in \mathfrak{g}_\mu^*$,
for all $\eta, \zeta_1, \zeta_2 \in \mathfrak{g}^\perp$.
This proves the Tube Condition.
\end{proof}

\medskip

\begin{remark} We note the three special cases that were combined in the 
``$(\dot \nu, \zeta) - (\dot \nu, \zeta)$'' case above:

\noindent
\underline{Case $\dot \nu - \dot \nu$}:
When $\dot \zeta_1 = \dot \zeta_2 = 0$, the condition in the proposition is equivalent to:
\begin{align*}
&\left\langle \mu + \nu,
\left[
F(\nu, \eta)^{-1} \left( DF(\nu, \eta) \cdot (\dot \nu_1, 0)\right),
F(\nu, \eta)^{-1} \left( DF(\nu, \eta) \cdot (\dot \nu_2, 0)\right)
\right]
\right\rangle \\
&\quad + 
\left\langle 
\dot \nu_2, 
F(\nu, \eta)^{-1} \left( DF(\nu, \eta) \cdot (\dot \nu_1, 0)\right) \right\rangle
- \left\langle 
\dot \nu_1, 
F(\nu, \eta)^{-1} \left( DF(\nu, \eta) \cdot (\dot \nu_2, 0)\right)
\right\rangle = 0.
\end{align*}
Note that a sufficient condition is that
$\left( DF(\nu, \eta) \cdot (\dot \nu_1, 0)\right)$ is a multiple of
$\left( DF(\nu, \eta) \cdot (\dot \nu_2, 0)\right)$
and 
$F(\nu, \eta)^{-1} \left( DF(\nu, \eta) \cdot (\dot \nu_1, 0)\right) \in \mathfrak{g}_\mu^\perp$
for all $\dot \nu_1,\dot \nu_2 \in \mathfrak{g}_\mu^*$.

\medskip
\noindent
\underline{Case $\dot \nu - \zeta$}:
When $\dot \zeta_1 = \dot \nu_2 = 0$, the condition in the proposition is equivalent to:
\begin{align*}
& - \left\langle (\mu + \nu),
\left[ 
F(\nu, \eta)^{-1} \left(  DF(\nu, \eta) \cdot (\dot \nu_1, 0)\right),
F(\nu, \eta)^{-1} \left(  DF(\nu, \eta) \cdot (0, \zeta_2)\right)
\right] 
\right\rangle \\
&\quad + \left\langle 
\dot \nu_1, 
F(\nu, \eta)^{-1} \left( DF(\nu, \eta) \cdot (0, \zeta_2)\right)
\right\rangle = 0
\end{align*}
all $\dot \nu_1, \zeta_2$,
since 
$\Omega_Y
(e, \nu, \ad\nolimits_\eta^* \mu)
\left(\left(0, \dot \nu_1, 0\right), \left(0, 0, \zeta_2\right)\right) = 0$.

\medskip
\noindent
\underline{Case $\zeta - \zeta$}:
When $\dot \nu_1 = \dot \nu_2 = 0$, the condition in the proposition is equivalent to:
\begin{align}\label{E:slicecond}
&\left\langle 
\mu + \nu, 
\left[ 
F(\nu, \eta)^{-1} \left(DF(\nu, \eta) \cdot \left(0, \zeta_1\right)\right),
F(\nu, \eta)^{-1} \left(DF(\nu, \eta) \cdot \left(0, \zeta_2\right)\right)
\right]
\right\rangle \\
&= \left\langle \mu, \left[\zeta_1, \zeta_2\right]\right\rangle. \notag
\end{align}

\end{remark}

The last case above may be compared with the following:

\medskip

\textbf{Restatement of Lemma \ref{preserveKKS}}:
Let $\varphi:T_\mu\mathcal{O}_\mu \to \mathcal{O}_\mu$ 
be of the form $\varphi(-\ad\nolimits^*_\eta \mu) = f(\eta)\mu$ for some
$f:\mathfrak{g}_\mu^\perp \to G$.
Then $\varphi$
preserves the $-$KKS symplectic form
if and only if
\begin{align}\label{E:preserveKKS2}
\left\langle\mu, \left[\zeta_1, \zeta_2\right]\right\rangle
&= \left\langle f(\eta)\mu, 
\left[\left(Df(\eta)\cdot \zeta_1\right) f(\eta)^{-1},
\left(Df(\eta)\cdot \zeta_2\right) f(\eta)^{-1} 
\right]
\right\rangle \\
&= \left\langle \mu, 
\left[f(\eta)^{-1} \left(Df(\eta)\cdot \zeta_1\right),
f(\eta)^{-1} \left(Df(\eta)\cdot \zeta_2\right)
\right]
\right\rangle \notag
\end{align}
for all $\eta,\zeta_1,\zeta_2 \in \mathfrak{g}_\mu^\perp$.

\begin{proof} 
$D\varphi(-\ad\nolimits^*_\eta \mu) \cdot \left(-\ad\nolimits^*_\zeta \mu\right)
= \left(Df(\eta) \cdot \zeta\right) \mu$ (using the ``hat'' map for $\mu$), 
which corresponds to 
\[
-\ad\nolimits^*_{\left(Df(\eta) \cdot \zeta_1\right)f(\eta)^{-1}} 
\left(\Ad\nolimits^*_{f(\eta)^{-1}} \mu\right),
\]
so
\begin{align*}
&\left(\varphi^*\Omega^-_{KKS}\right) \left(-\ad\nolimits_\eta^*\mu\right)
\left(-\ad\nolimits_{\zeta_1}^*\mu, -\ad\nolimits_{\zeta_2}^*\mu\right) \\
&\quad = \Omega^-_{KKS} (f(\eta) \mu)
\left(\left(Df(\eta) \cdot \zeta_1\right) \mu, \left(Df(\eta) \cdot \zeta_2\right) \mu\right) \\
&\quad = \Omega^-_{KKS} (\Ad\nolimits^*_{f(\eta)^{-1}} \mu)
\left(
\ad\nolimits^*_{\left(Df(\eta) \cdot \zeta_1\right)f(\eta)^{-1}} 
\left(\Ad\nolimits^*_{f(\eta)^{-1}} \mu\right), 
\ad\nolimits^*_{\left(Df(\eta) \cdot \zeta_2\right)f(\eta)^{-1}}
\left(\Ad\nolimits^*_{f(\eta)^{-1}} \mu\right)
\right) \\
&\quad = 
\left\langle 
\Ad\nolimits^*_{f(\eta)^{-1}} \mu,
\left[
\left(Df(\eta) \cdot \zeta_1\right)f(\eta)^{-1} ,
\left(Df(\eta) \cdot \zeta_2\right)f(\eta)^{-1} 
\right]
\right\rangle \\
&\quad = 
\left\langle 
\mu,
\left[
\Ad\nolimits_{f(\eta)^{-1}} \left(\left(Df(\eta) \cdot \zeta_1\right)f(\eta)^{-1}\right) ,
\Ad\nolimits_{f(\eta)^{-1}} \left(\left(Df(\eta) \cdot \zeta_2\right)f(\eta)^{-1}\right) 
\right]
\right\rangle \\
&\quad = 
\left\langle 
\mu,
\left[
f(\eta)^{-1} \left(Df(\eta) \cdot \zeta_1\right) ,
f(\eta)^{-1} \left(Df(\eta) \cdot \zeta_2\right)
\right]
\right\rangle
\end{align*}
\end{proof}

The similarity of conditions (\ref{E:slicecond}) and (\ref{E:preserveKKS2})
led to the discovery of an explicit construction of a symplectic tube for $G=SO(3)$, see 
Section \ref{ssect:sliceSO3}.

\medskip

Finally, we prove Lemma \ref{lem:ansatz} in Section \ref{ssect:sliceSO3}.
This lemma concerns an Ansatz that is motivated by our consideration of the $SO(3)$ case. 
However, the lemma is valid for all Lie groups.

\textbf{Restatement of Lemma \ref{lem:ansatz}}:
Suppose $\displaystyle F(\nu, \eta) = \exp\left(h(\nu,\eta) \,  \frac{\eta}{\|\eta\|}\right)$,
for some $h: \lieg_{\mu}^* \times \lieg_\mu^\perp \to \R$. 
Then the Tube Condition in Proposition \ref{tubecond} is automatically satisfied
(regardless of the definition of $h$) for all 
$\left(\dot \nu_1, \zeta_1\right), \left(\dot \nu_2, \zeta_2\right)$ such that
$\zeta_1$ and $\zeta_2$ are parallel to $\eta$.

\begin{proof}
For arbitrary $f:\R \to M(n,\R)$, if $f'(0)$ is a multiple of $f(0)$ then they commute, so
\[
\left.\frac{d}{dt}\right|_{t=0} \exp \left(f(t)\right) 
= \left.\frac{d}{dt}\right|_{t=0} \exp \left(f(0) + t f'(0)\right) 
= \exp \left(f(0)\right) f'(0) 
= f'(0) \exp \left(f(0)\right).
\]
Let $f(t) = h\left((\nu, \eta) + t(\dot \nu, \zeta)\right) \frac{\eta}{\|\eta\|}$.
If $\zeta$ is parallel to $\eta$ then 
$\displaystyle \frac{\eta + t \zeta}{\|\eta + t \zeta\|} = \frac{\eta}{\|\eta\|}$, so
\begin{align*}
DF(\nu, \eta) \cdot (\dot \nu, \zeta)
&= \left.\frac{d}{dt}\right|_{t=0}
F\left((\nu, \eta) + t (\dot \nu, \zeta)\right)\\
&= \left.\frac{d}{dt}\right|_{t=0}
\exp\left(h\left((\nu,\eta) + t (\dot \nu, \zeta)\right) \,  \frac{\eta}{\|\eta\|}\right) \\
&= \left.\frac{d}{dt}\right|_{t=0}
\exp \left(f(t)\right) \\
&= f'(0) \exp \left(f(0)\right) \\
&= \left(Dh(\nu, \eta) \cdot (\dot \nu, \zeta)\right) \frac{\eta}{\| \eta \|} F(\nu, \eta) \\
&=  F(\nu, \eta) \left(Dh(\nu, \eta) \cdot (\dot \nu, \zeta)\right) \frac{\eta}{\| \eta \|}.
\end{align*}
Thus if $\zeta_1$ and $\zeta_2$ are parallel to $\eta$, 
\begin{align*}
\left[ 
F(\nu, \eta)^{-1} \left(DF(\nu, \eta) \cdot \left(\dot \nu_1, \zeta_1\right)\right),
F(\nu, \eta)^{-1} \left(DF(\nu, \eta) \cdot \left(\dot \nu_2, \zeta_2\right)\right)
\right] = 0
\end{align*}
and 
\begin{align*}
\left\langle \dot \nu_i, F(\nu, \eta)^{-1} \left(DF(\nu, \eta) \cdot \left(\dot \nu_j, \zeta_j\right)\right)
\right\rangle
= 0,
\end{align*}
for all $i,j = 1,2$.
Therefore the Tube Condition in Proposition \ref{tubecond} holds.
\end{proof}

\bibliographystyle{spmpsci}      
\bibliography{tanya}    

 \end{document}